\documentclass[journal,twoside,web]{ieeecolor}
\usepackage{generic}
\usepackage{cite}
\usepackage{amsmath,amssymb,amsfonts}
\usepackage{algorithmic}
\usepackage{graphicx}
\usepackage{algorithm,algorithmic}
\usepackage{hyperref}
\usepackage{amssymb}
\usepackage{verbatim}
\usepackage{mathrsfs}

  \makeatletter
  \let\NAT@parse\undefined
  \makeatother

\usepackage{tikz}
\usetikzlibrary{arrows, shapes, chains}
\tikzstyle{startstop} = [rectangle,rounded corners, minimum width=3cm,minimum height=1cm,text centered, draw=black,fill=red!30]
\tikzstyle{io} = [trapezium, trapezium left angle = 70,trapezium right angle=110,minimum width=3cm,minimum height=1cm,text centered,draw=black,fill=blue!30]
\tikzstyle{process} = [rectangle,minimum width=5cm,minimum height=2cm,text centered, text width =5cm,draw=black,fill=white]
\tikzstyle{decision} = [diamond,minimum width=3cm,minimum height=1cm,text centered,draw=black,fill=green!30]
\tikzstyle{arrow} = [thick,->,>=stealth]
\usepackage{subcaption}
\usepackage{amsthm}
  \theoremstyle{plain}
  \newtheorem{theorem}{Theorem}
  \newtheorem{lemma}{Lemma}

  \newtheorem{assumption}{Assumption}
  
  \newtheorem{example}{Example}
  
  \theoremstyle{remark}
  
\hypersetup{hidelinks=true}

\DeclareMathOperator*{\argmin}{arg\,min}

\newcommand{\R}{\mathbb{R}}
\newcommand{\w}{\omega}
\usepackage{textcomp}
\def\BibTeX{{\rm B\kern-.05em{\sc i\kern-.025em b}\kern-.08em
    T\kern-.1667em\lower.7ex\hbox{E}\kern-.125emX}}
\allowdisplaybreaks[4]
\begin{document}
\title{{Douglas-Rachford Splitting for Group-Sparse Feedback Linear-Quadratic Control}}
\author{Lechen Feng, Xun Li, Yuan-Hua Ni
\thanks{Lechen Feng is with the Department of Applied Mathematics, The Hong Kong Polytechnic University, China. Email: {\tt fenglechen0326@163.com}.
Xun Li is with the Department of Applied Mathematics, The Hong Kong Polytechnic University, China. Email: {\tt li.xun@polyu.edu.hk}. This author is supported by the Research Grants Council (RGC) of Hong Kong under Grants 15226922 and 15225124, and partially supported by projects 4-ZZVB and PolyU 4-ZZP4.
Yuan-Hua Ni is with the College of Artificial Intelligence, Nankai University, Tianjin, China Email: {\tt yhni@nankai.edu.cn}.}
}

\maketitle

\begin{abstract}
In this paper, we study the distributed linear quadratic problem with fixed communication topology (DFT-LQ) and the sparse feedback linear quadratic (SF-LQ) problem through a unified optimization framework. 
Specifically, both problems are formulated as a nonconvex, nonsmooth optimization problem equipped with an $\ell_0$-penalty under affine constraints. To solve this problem, we first investigate the application of the Douglas-Rachford (DR) splitting algorithm. Under the local condition that the generated iterates remain on a fixed smooth manifold, we establish the convergence of the DR splitting to a stationary point. Furthermore, we characterize this stationary point as the global minimizer of a corresponding DFT-LQ problem. To bypass the restriction of the smooth manifold assumption, we introduce a projected subgradient descent algorithm that achieves global convergence without relying on smooth-manifold structures. This algorithm may serve as a warm-start mechanism that effectively drives the iterates toward the desired smooth manifolds, thereby establishing a favorable initialization where the convergence theory of the DR splitting algorithm becomes fully applicable. Numerical experiments shed light on the effectiveness of the proposed methods in distributed group-sparse controller design.

\end{abstract}

\begin{IEEEkeywords}
Linear-quadratic problem,  sparse feedback, constrained optimization, DR splitting
\end{IEEEkeywords}

\section{Introduction}

During the last few decades, distributed control has attracted significant attention due to its appealing features, including reduced communication costs and parallelizable computation.  In this paper, we study the distributed LQ problem with fixed communication topology (DFT-LQ) and the sparse feedback LQ (SF-LQ) problem. Specifically, we establish an optimization framework for designing group-sparse feedback gains that minimize the $\mathcal{H}_2$ performance cost of distributed systems. Our approach consists of three steps. First, we reformulate the SF-LQ and DFT-LQ problems as a nonsmooth and nonconvex affine-constrained optimization problem with an $\ell_0$-penalty. Second, we show that the Douglas-Rachford (DR) splitting algorithm converges to a stationary point of the resulting problem, under the assumption that the iterates remain within a fixed smooth manifold. Surprisingly, we further prove that such a stationary point corresponds to a global minimizer of a certain DFT-LQ problem. Finally, we develop a globally convergent algorithm that may serve as a warm-start procedure for DR splitting by steering the iterates toward a smooth manifold. Ultimately, this work aims to facilitate the establishment of effective coordination among multiple sensors or decision-makers using the minimum possible information exchange. Even in scenarios where communication resources are not scarce, a profound understanding of these fundamentals (and the elucidation of the inherent fundamental trade-offs) is instrumental in constructing a desirable system architecture, which ultimately helps reduce both computational requirements and overall complexity.

\subsection{Related Works}
\textbf{1) DFT-LQ Problem.}
The DFT-LQ problem presents significant challenges, while Tsitsiklis and Athans (1985) \cite{Tsitsiklis-1985} pointed out the inherent difficulty of decentralized decision making and suggest that optimality may be an elusive goal. Consequently, three mainstream research paradigms have emerged for tackling the DFT-LQ problem. The first line of research reformulates the DFT-LQ problem into an infinite-dimensional convex problem and utilize finite-dimensional approximations (e.g., Ritz approximations) to solve it; see \cite{rotkowitz2005characterization,furieri2020sparsity} and the references therein. The second line of research focuses primarily on fully decentralized DFT-LQ problems. By relaxing the DFT-LQ problem into a finite-dimensional convex optimization problem, this approach enables the use of existing convex optimization techniques for controller synthesis; see \cite{i14,ref1,i9,i12} for the details. The third line of research treats the DFT-LQ  problem directly as a nonconvex optimization problem and develops nonconvex optimization algorithms for solving it; e.g., rank-$1$ semi-definite program (SDP) algorithm \cite{i10}, zero-order distributed policy optimization algorithm \cite{i11} and manifold Newton-type algorithm \cite{talebi2022policy}. Nevertheless, the above three research paradigms suffer from the following intrinsic limitations, respectively:
\begin{itemize}
  \item For the first line of research, the infinite-dimensional problem cannot be solved and therefore necessarily relies on finite-dimensional approximations, which may introduce non-negligible and uncontrollable approximation errors.
  \item For the second line of research, the resulting convex optimization framework is mainly restricted to fully decentralized communication structures and therefore cannot accommodate general DFT-LQ problems.
  \item For the third line of research, existing algorithms are typically heuristic or only admit local convergence guarantees to stationary points, without a precise characterization of the resulting stationary solutions.
\end{itemize}

\textbf{2) SF-LQ Problem.}
The SF-LQ problem exhibits the following non-constrained reformulation
\begin{equation}\label{l0-intro}
  \min_{K\in\mathcal{S}}~J(K)+\gamma \Vert K\Vert_0
\end{equation}
with LQ cost $J(\cdot)$, stabilizing feedback set $\mathcal{S}$ and constant $\gamma>0$. Due to the nonconvex nature of both the objective function $J(K)$ and the feasible set $\mathcal{S}$, existing sparse optimization algorithms are not directly applicable. Therefore, it is necessary to design specialized algorithms tailored to problem \eqref{l0-intro}.
Specifically, Lin \textit{et al.} (2013) \cite{kk4} utilized Alternating Direction Method of Multipliers (ADMM) to solve problem \eqref{l0-intro}, although no theoretical convergence analysis has been presented.
Bahavarnia (2015) \cite{bahavarnia2015sparse} and Park \textit{et al.} (2020) \cite{park2020structured} investigated the convex relaxation of problem \eqref{l0-intro}:
\begin{equation}\label{l1-intro}
  \min_{K\in\mathcal{S}}~J(K)+\gamma \Vert K\Vert_1,
\end{equation}
and proved that the Structured Policy Iteration algorithm can solve problem \eqref{l1-intro} with a convergence guarantee to a stationary point.
Following the approach of $\ell_1$ relaxation, Takakura and Sato (2024) \cite{takakura2023structured} revisited the policy gradient algorithm based on the gradient projection method and showed its global convergence to $\epsilon$-stationary points of the static output SF-LQ problem.
More generally, Negi and Chakrabortty (2020) \cite{negi2020sparsity} studied the generalization of problem \eqref{l1-intro} with time-delays in the feedback and fair sharing of bandwidth among users.

\textbf{3) First-Order Nonconvex Optimization.} The following standard problem
\begin{align}
\min_{x,y}~~&F(x)+G(y),~~\mathrm{s.t.}~~Ax+By=c,\label{nonconvex-2}
\end{align}
is well-studied in the existing literature, where $A, B, c$ are given matrices.
Concretely, ADMM is the most widely adopted and general-purpose methodology; see \cite{hong2016convergence,themelis2020douglas,li2015global} and the references therein for the details. 
In fact, both the DFT-LQ and SF-LQ problems can be reformulated into the form of problem \eqref{nonconvex-2}, where both $F$ and $G$ are nonsmooth. In particular, $G$ typically corresponds to a group $\ell_0$-penalty. However, Barber and Sidky (2024) \cite{barber2024convergence} observed that essentially all existing convergence analyses of ADMM have fallen into one of the following two categories:
\begin{itemize}
  \item Assume that either $F$ or $G$ is $L$-smooth and establish the convergence guarantees.
  \item Assume that the algorithm converges and establish the optimality properties of limit point.
\end{itemize}
Since the aforementioned works generally require at least one of $F$ or $G$ to be smooth, the jointly nonsmooth setting considered in this paper lies beyond the scope of the existing ADMM theory.

\subsection{Contributions}\label{contribution-sec}

Compared with existing results, the main contributions of the presented paper are as follows.

\textbf{1) Contributions for LQ Control.} 

\begin{itemize}
    \item From the epi-composition function perspective, We thoroughly investigate the optimization landscapes of the SF-LQ and DFT-LQ problems under the classical parameterization method, which was pioneered in \cite{i14} and has been widely adopted in studies such as \cite{ref1, i9, i12}. In particular, we analyze the subgradient properties, effective domain structures, and $L$-smoothness of the corresponding optimization problem when an additional $\ell_0$-penalty is introduced; see Lemma \ref{subgrad_varphi1} and Lemma \ref{chan}. These insights shed light on and pave the way for investigating the SF-LQ and DFT-LQ problems via nonconvex optimization algorithms.
  \item For the SF-LQ problem in particular, the $\ell_0$-penalty-induced optimization problem is of primary interest. We reveal a novel mechanism induced by the $\ell_0$-penalty for the SF-LQ problem, under which any stationary point of the SF-LQ problem corresponds to a global minimizer of a certain DFT-LQ problem; see Theorem \ref{local-minimal}. To the best of our knowledge, this is the first work to uncover such an intrinsic connection between the SF-LQ and DFT-LQ problems.
\end{itemize}

\textbf{2) Contributions for Nonconvex Optimization.}  This paper studies the DFT-LQ and SF-LQ problems through a nonsmooth and nonconvex optimization problem. Compared with existing literature, this paper offers the following distinct contributions.
\begin{itemize}
  \item We prove the convergence of DR splitting to a stationary point of the aforementioned problem, provided that the iterates remain within a fixed smooth manifold; see Theorem \ref{smoothDR}. Moreover, we show that these stationary points admit a clear control-theoretic interpretation: each of them corresponds to a global minimizer of a certain DFT-LQ problem. This significantly addresses a major limitation of existing papers, where the stationary points obtained by the proposed algorithms typically lack any clear control-theoretic interpretation or performance guarantees; see \cite{i10,i11,talebi2022policy,kk4,takakura2023structured,negi2020sparsity,bahavarnia2015sparse,park2020structured} for the details.
  \item We then introduce a difference-of-convex (DC) relaxation of the $\ell_0$-penalty and show that the relaxed formulation preserves the variational structure of the original $\ell_0$-induced problem; see Theorem \ref{thm_app}. We further establish the global convergence of the Projected Subgradient Descent (PSGD) algorithm applied to the relaxed problem (see Theorem \ref{thm7}), thereby providing a warm-start strategy for DR splitting. 
      In contrast to existing PSGD analyses that establish only Clarke stationarity (e.g., \cite{davis2020stochastic}), we demonstrate the convergence to stationary points characterized by the limiting subdifferential.
\end{itemize}


{\textbf{Notation}}. Let $\Vert \cdot\Vert$, $\Vert\cdot\Vert_F$ and $\Vert\cdot\Vert_0$ be the spectral norm, Frobenius norm and $\ell_0$-norm of a matrix respectively. $\mathbb{S}^n$ is the set of symmetric matrices of $n\times n$;  $\mathbb{S}^n_{++}$ ($\mathbb{S}^n_+$)  is the set of positive (semi-)definite (PSD) matrices of $n\times n$; $I_n$ is identity matrix of $n\times n$; $\mathbf{0}_n$ is $(0,\dots,0)^\top\in\mathbb{R}^n$ and $\mathbf{1}_n$ is $(1,\dots,1)^\top\in\mathbb{R}^n$.  $A\succ B(A\succeq B)$ means that the matrix $A-B$ is positive (semi-)definite.
$\mathrm{vec}(A)$ denotes the column vector formed by stacking columns of $A$ one by one, and
$\Gamma_{+}^n$ is $\{\mathrm{vec}(A)\colon A\in\mathbb{S}_+^n\}$.
The operator $\langle A,B\rangle$ denotes the Frobenius inner product, i.e., $\langle A,B\rangle=\mathrm{Tr}(A^\top B)$ for all $A,B\in\mathbb{R}^{m\times n}$, and the notation $\otimes$ denotes the Kronecker product of two matrices.
For $\tau>0$, introduce the proximal operator of $g\colon \mathbb{R}^n\to\mathbb{R}$
\begin{equation*}
  \mathbf{prox}_{\tau g}(z):=\underset{y\in\mathbb{R}^n}{\operatorname*{argmin}}\left\{g(y)+\frac{1}{2\tau}\left\|y-z\right\|^2\right\}.
\end{equation*}
For any subset $C\subseteq\mathbb{R}^n$, $\delta_{C}(\cdot)$ is the indicator function of $C$, while we denote $\mathbf{Proj}_{C}(x)\doteq \mathbf{Prox}_{\delta_C}(x)$ for any $x\in\mathbb{R}^n$.
For $n\geq 1$, the set $[n]$ denotes $\{1,\dots,n\}$. 
Given a vector $x\in\mathbb{R}^n$, the subvector of $x$ composed of the components of $x$ whose indices are in a given subset $\mathbb{I}\subseteq [n]$ is denoted by $x_\mathbb{I}\in\mathbb{R}^{|\mathbb{I}|}$.
For an extended real-valued function $f\colon \mathbb{R}^n\to [-\infty,+\infty]$, the effective domain is the set $\mathrm{dom}(f)=\{x\in\mathbb{R}^n\colon f(x)<\infty\}$.
We denote $\operatorname{ri}(C)$ as the relative interior of $C$.
For any $\alpha\in\mathbb{R}$, the $\alpha$-level set of a function $f:\mathbb{R}^n\to[-\infty,\infty]$ is the set
$\mathbf{lev}_{\leq\alpha}f=\{x\in\mathbb{R}^n\colon f(x)\leq\alpha\}.$ 
$\widehat{\partial}f(x)$, $\partial f(x)$, $\partial^\infty f(x)$ and $\partial^C f(x)$
denote the Fr\'echet, limiting, horizon, and Clarke subdifferentials of \(f\) at \(x\), respectively; see Section 8 of \cite{rockafellar-book} for the details. $\mathcal{N}_C(x)\doteq \partial \delta_C(x)$ denotes the limiting normal cone to \(C\) at \(x\).
Given $h\colon\mathbb{R}^n\to\bar{\mathbb{R}}$ and $F\in\mathbb{R}^{m\times n}$, the epi-composition function $(Fh)\colon \mathbb{R}^m\to\bar{\mathbb{R}}$ is defined as
$
  (Fh)(u)\doteq \inf\{ h(x)\colon Fx=u\}.
$

Consider a partitioned matrix
$
  A=\left(\begin{smallmatrix}
      A_{11} & \cdots & A_{1t} \\
      \vdots & \ddots & \vdots \\
      A_{s1} & \cdots & A_{st}
    \end{smallmatrix}\right)\in\mathbb{R}^{m\times n}
$,
where block $A_{ij}\in\mathbb{R}^{m_i\times n_j}$ for $i\in [s]$ and $j\in[t]$. We denote such a partitioned matrix as $A=\mathbf{block}_{s,t}(A_{11},\dots,A_{st})$. Define a mapping $g\colon \mathbb{R}^{m\times n}\to \{0,1\}^{s\times t}$ that maps a partitioned matrix $A$ to a binary matrix indicating the presence of nonzero blocks:
\begin{equation*}
  g(A)_{ij}=\left\{
  \begin{aligned}
  &1,~~A_{ij} \neq 0,\\
  &0,~~A_{ij} = 0.
  \end{aligned}
  \right.
\end{equation*}
Based on this mapping, we define the group $\ell_0$-norm of a partitioned matrix $A$ as $\Vert A\Vert_{s,t; A_{11},\dots,A_{st};0}=\Vert g(A)\Vert_0$, which counts the number of nonzero blocks in $A$. For brevity, we write this as $\Vert A\Vert_{s,t;0}$ when the block structure is clear from the context. 
For notational convenience, we identify the block index $(i,j)\in [s]\times[t]$ 
with the scalar index $\ell = i + s(j-1)\in [st]$, corresponding to the 
column-major ordering consistent with the vectorization of $P$. 
Accordingly, all block-dependent vectors are arranged following this ordering.
We define $\pi(\mathrm{vec}(P)) \in \{0,1\}^{st}$ as the binary vector encoding 
the block sparsity pattern of $P$, and define 
$\widetilde{\pi}(\mathrm{vec}(P))\in\mathbb{R}^{st}_+$ as the vector collecting 
the squared Frobenius norms of the blocks of $P$.

\section{Preliminaries}\label{section2}

Consider a linear time-invariant (LTI) system
\begin{equation}\label{system}
\begin{aligned}
\dot{x}(t)&=Ax(t)+B_2u(t)+B_1w(t),\\
z(t)&=Cx(t)+Du(t)
\end{aligned}
\end{equation}
with state $x(t)\in\mathbb{R}^n$, input $u(t)\in\mathbb{R}^m$, exogenous disturbance input $w(t)\in\mathbb{R}^l$, controlled output $z(t)\in\mathbb{R}^q$ and matrices $A,B_2,B_1,C,D$ with proper sizes. The infinite-horizon LQ problem is to find a linear static state feedback gain $K$ such that
$u(t)=-Kx(t)$ solves
\begin{align*}
 \min_{K}~\int_{0}^{\infty}z(t)^\top z(t){\rm d}t.
\end{align*}
Under $C^\top D=0$ and $D^\top D \succ 0$, the above problem is equivalent to
\begin{align*}
 \min_{K}~J(K)\doteq{\rm Tr}((C-DK)W_c(C-DK)^\top),
\end{align*}
where $W_c$ is the controllability Gramian associated with the closed-loop system; see Section 4 of \cite{i14} for the details.
Letting $\mathcal{S}$ denote the set of controllers that internally stabilize system (\ref{system}), we look at solving the following two optimization problems:
\begin{align}
  &\min_{K\in\mathcal{S}}~J(K),~{\rm s.t.}~K\in{\mathcal{K}},~\text{(DFT-LQ)}, \label{intro-1}\\
  &\min_{K\in\mathcal{S}}~J(K),~{\rm s.t.}~K\text{~is group sparse},~\text{(SF-LQ)}, \label{intro-2}
\end{align}
where $\mathcal{K}\doteq \{K\in\mathbb{R}^{m\times n}\colon K_{ij}=0,~\text{for all}~(i,j)\in\mathbf{U}\}$ with predetermined subset $\mathbf{U}\subseteq [m]\times [n]$.

\subsection{SF-LQ Problem}

Denote $F=\left(\begin{smallmatrix}
      A & -B_2 \\
      \mathbf{0}_{m\times n} & \mathbf{0}_{m\times m}
                \end{smallmatrix}\right)$, $G=\left(\begin{smallmatrix}
        \mathbf{0}_{n\times m}\\
        I_m
                \end{smallmatrix}\right)$, $Q=\left(\begin{smallmatrix}
B_1B_1^\top & \mathbf{0}_{n\times m} \\
        \mathbf{0}_{m\times n} & \mathbf{0}_{m\times m}
                \end{smallmatrix}\right)$ and $R=\left(\begin{smallmatrix}
C^\top C & \mathbf{0}_{n\times m} \\
          \mathbf{0}_{m\times n} & D^\top D
                \end{smallmatrix}\right)$, while let block matrix
$
W=\left(\begin{smallmatrix}
W_1 & W_2\\
W_2^\top & W_3
\end{smallmatrix}\right)\in\mathbb{S}^p
$
with $p\doteq m+n$, $W_1\in\mathbb{S}_{++}^n$, $W_2\in\mathbb{R}^{n\times m}$ and $W_3\in\mathbb{S}^m$. We introduce the following two assumptions, which are standing assumptions throughout this paper.

\begin{assumption}\label{ass1}
Let $C^\top D=0$, $D^\top D\succ 0$, $C^\top C\succ 0$, $B_1B_1^\top\succ0$, $(A,B_2)$ be stabilizable and $(A,C)$ have no unobservable modes on the imaginary axis.
\end{assumption}

\begin{assumption}\label{ass2}
The parameter $F$ is unknown but convex-bounded, i.e., $F$ belongs to a polyhedral domain, which is expressed as a convex combination of the extreme matrices: $F=\sum_{i=1}^{M}\xi_iF_i,\xi_i\geq 0,\sum_{i=1}^{M}\xi_i=1$, and $F_i=\left(\begin{smallmatrix}
                              A_i & -B_{2,i} \\
                              0 & 0
                            \end{smallmatrix}\right)\in\mathbb{R}^{p\times p}$
denotes the extreme vertex of the uncertain domain.
\end{assumption}

By introducing group $\ell_0$-norm into objective function of \eqref{intro-2}, the SF-LQ problem can be reformulated by 
\begin{equation}\label{sparse_feedback_1}
	\begin{aligned}
            \min_{K\in\mathcal{S}}~ J(K)+\gamma\Vert K\Vert_{s,t;K_{11},\dots,K_{st};0},
    \end{aligned}
\end{equation}
where $K_{ij}\in\mathbb{R}^{m_i\times n_j}$ represents the block components of feedback gain for any $i\in[s]$ and $j\in[t]$, and $\gamma\geq 0$ is a weighting parameter. 

\begin{lemma}[\!\!\cite{ref1}]\label{thm_para}
  Under Assumption \ref{ass1} and \ref{ass2}, one can define the set
$
  \mathscr{C}=\{W\in\mathbb{S}^p_+:A_iW_1-B_{2,i}W_2^\top+W_1A_i-W_2B_{2,i}+B_1B_1^\top\preceq0,~\forall i\in[M]\},
$
and let
    $\mathscr{K}=\{K=W_2^\top W_1^{-1}\colon W\in\mathscr{C}\}$.
  Then, $K\in\mathscr{K}$ stabilizes the closed-loop system. Moreover, $K\in\mathscr{K}$ gives
    $\langle R,W\rangle\geq \Vert H_i(s)\Vert_{\mathcal{H}_2}^2$ for any $i\in[M]$,
    where $\Vert H_i(s)\Vert_{\mathcal{H}_2}$ represents the $\mathcal{H}_2$-norm with respect to the $i$-th extreme system. Furthermore, in the absence of system uncertainty (i.e., $M=1$), the aforementioned inequality reduces to an equality, exactly yielding $\langle R,W\rangle = \Vert H(s)\Vert_{\mathcal{H}_2}^2$.
\end{lemma}

Based on Lemma \ref{thm_para}, we reformulate  problem \eqref{sparse_feedback_1} as
\begin{equation}\label{sparse_feedback_2}
         \begin{aligned}
           \min\limits_{W\in\mathscr{C} \atop K\in\mathbb{R}^{m\times n}}&\langle R,W\rangle+\gamma \Vert K\Vert_{s,t;0}~~{\rm s.t.}~~K=W_2^\top W_1^{-1}.
         \end{aligned}
\end{equation}
Due to the nonlinear manifold constraint $K=W_2^\top W_1^{-1}$, optimization problem \eqref{sparse_feedback_2} exhibits substantial challenges. 
To bypass this difficulty, a widely adopted strategy is to restrict $W_1$ to a specific structural pattern so that the sparsity of $K$ can be directly controlled by $W_2^\top$; see \cite{rotkowitz2005characterization,furieri2020sparsity,i14,ref1,i9,i12}. Specifically, for any binary-valued block matrix $X$, define $\mathbf{Sparse}(X)$ as the set of block matrices sharing the same zero block pattern as $X$. If
\begin{equation}\label{W1linear}
  W_1\in \mathbf{Sparse}(\mathbf{blockdiag}(I_{n_1},\dots,I_{n_t})),
\end{equation}
then the feedback gain $K=W_2^\top W_1^{-1}$ inherits the sparsity pattern of $W_2^\top$. Consequently, minimizing $\Vert K\Vert_{s,t;0}$ can be relaxed to minimizing $\Vert W_2^\top\Vert_{s,t;0}$ under the additional linear constraint on $W_1$.
Hence, denoting $N\doteq t(t-1)/2$,
problem \eqref{sparse_feedback_2} can be relaxed to
\begin{align}
  \min_{W,P}~~&\langle R,W\rangle+\gamma \Vert P\Vert_{s,t;0}\notag\\
  {\rm s.t.}~~&W\in\mathbb{S}^p_+,\notag\\
  &\Psi_i\doteq-V_2(F_iW+WF_i^\top+Q)V_2^\top\in\mathbb{S}_+^{n},~i\in[M],\notag\\
  &-V_{j1}WV_{j2}=0,~j\in[N],\label{sparse_feedback_4}\\
  &V_1WV_2^\top-P=0\notag
\end{align}
with $\gamma>0$, $V_1=[0,I_m],V_2=[I_n,0]$. 
Here, matrices \(V_{j1}\) and \(V_{j2}\) are constructed to enumerate all off-diagonal blocks of \(W_1\); therefore, for any $j\in[N]$, constraint
$
V_{j1}WV_{j2}=0
$
is equivalent to the block-diagonal constraint \eqref{W1linear} on $W_1$.
Let $\widetilde{W}={\rm vec}(W),\widetilde{P}={\rm vec}(P), \widetilde{\Psi}_i={\rm vec}(\Psi_i)$ for $i\in[M]$,
and define 
$
  \mathcal{A}\doteq\left( \begin{smallmatrix}
    \widehat{\mathcal{A}}\\
    V_2\otimes V_1
  \end{smallmatrix}\right)$ and
  $\mathcal{B}=\left(\begin{smallmatrix}
  \mathbf{0}_{N^*}\\
  -I_{mn}
  \end{smallmatrix}\right)$,
where $\widehat{\mathcal{A}}\doteq \left(\begin{smallmatrix}
                                           V_{12}^\top\otimes V_{11} \\
                                           \vdots \\
                                           V_{N2}^\top\otimes V_{N1} 
                                         \end{smallmatrix}\right)\in\mathbb{R}^{N^*\times p^2}$ 
 and $N^*$ represents the total number of rows in $\widehat{\mathcal{A}}$ ($N^*$ can be directly computed from the problem dimensions).
By the definition of the indicator function, problem \eqref{sparse_feedback_4} is further equivalently to the standard form:
\begin{equation}\label{sparse_feedback_6}
\begin{aligned}
\min_{\widetilde{W},\widetilde{P}}~~ &f(\widetilde{W})+g(\widetilde{P})~~{\rm s.t.}~~\mathcal{A}\widetilde{W}+\mathcal{B}\widetilde{P}=0
\end{aligned}
\end{equation}
with
\begin{align*}
&f(\widetilde{W})=\langle {\rm vec}(R),\widetilde{W}\rangle+\delta_{\Gamma_+^p}(\widetilde{W})+\sum_{i=1}^{M}\delta_{\Gamma_+^n}(\widetilde{\Psi}_i),\\
&g(\widetilde{P})=\gamma \Vert \mathrm{vec}^{-1}(\widetilde{P})\Vert_{s,t; 0},
\end{align*}
where $\mathrm{vec}^{-1}(\cdot)$ denotes the inverse operator of $\text{vec}(\cdot)$.

\subsection{DFT-LQ Problem}\label{section_link}

In this subsection, we establish the relationship between the SF-LQ and DFT-LQ problems, demonstrating that SF-LQ problems can be viewed as a generalization of the DFT-LQ problem. This connection allows us to solve the DFT-LQ problem by solving problem \eqref{sparse_feedback_6}.
Generally, a group-sparse feedback gain $K$ can lead to a multi-agent distributed control system. For example, given a group-sparse feedback gain $K=\mathbf{block}_{s,t}(K_{11},\dots,K_{st})$, it can induce a $s$-agents distributed system. Specifically, the global control action is composed of local control actions: $u(t)=[u_1(t)^\top,\dots,u_s(t)^\top]^\top$, where $u_i(t)$ is the control input of
  agent $i$. At time $t$, agent $i$ directly observes a partial state $x_{\mathcal{I}_i}(t)$ with
  \begin{equation*}
    \mathcal{I}_i=\bigcup_{j=1,\dots,t,~K_{ij}\neq 0}\left\{n_{j-1}+1,\dots,n_j\right\}\subseteq [n], \text{~and~}n_0\doteq0.
  \end{equation*}
Similarly, DFT-LQ problem \eqref{intro-1} can be relaxed into the following form
  \begin{align*}
  \min_{W,P}~\langle R,W\rangle~~{\rm s.t.}~~\mathcal{G}(W)\in\mathcal{K},~[W_2^\top]_{ij}=0,~(i,j)\in\mathbf{U}
  \end{align*}
with predefined $\mathbf{U}\subseteq [m]\times [n]$, which can be regarded as a modification of problem \eqref{sparse_feedback_6}.
In summary, the purpose of this subsection is merely to illustrate that the DFT-LQ problem can also be formulated by the form of \eqref{sparse_feedback_6}. From this perspective, the SF-LQ problem provides a unified optimization framework that naturally incorporates distributed controller synthesis with prescribed communication topology. Moreover, the DFT-LQ problem admits more general extensions beyond the formulation considered here. Due to space limitations, we refer interested readers to Section 2.2 of \cite{feng-part1} for a comprehensive discussion.

\section{DR Splitting Algorithm}

Obviously, by introducing an augmented variable $s$ and the definition of the epi-composition function, optimization problem \eqref{sparse_feedback_6} is equivalent to  
\begin{equation}\label{admm-standard}
  \min_{s} ~~\varphi_1(s)+\varphi_2(s)
\end{equation}
with $\varphi_1(s)\doteq(\mathcal{A}f)(s)$ and $\varphi_2(s)\doteq(\mathcal{B}g)(-s)$.




\subsection{Optimization Landscape of Optimization Problem \eqref{admm-standard}}\label{landscape}

\begin{lemma}\label{subgrad_varphi1}
Under Assumption \ref{ass1} and \ref{ass2}, the following results hold.
\begin{enumerate}
	\item $\varphi_1(s)$ is closed and convex.
	\item Denote set
$
  \widetilde{W}(\hat{s})=\{\widetilde{W}\in\mathbb{R}^{p^2}\colon \mathcal{A}\widetilde{W}=\hat{s},f(\widetilde{W})=\varphi_1(\hat{s})\},
$
and assume that $\widetilde{W}(\hat{s})\neq\emptyset$. Then, for arbitrary $\widetilde{W}\in\widetilde{W}(\hat{s})$, it follows that
\begin{align*}
  \partial \varphi_1(\hat{s})&{\supseteq}\bigg\{s\in\mathbb{R}^{N^*+mn}\colon \mathcal{A}^\top s\in\mathrm{vec}(R)\\
  &~~~~+\mathcal{N}_{\Gamma_+^p}(\widetilde{W})+\sum_{i=1}^{M}\mathcal{N}_{\mathcal{H}_i^{-1}(\Gamma_+^n)}(\widetilde{W})\bigg\}.
\end{align*}
\item $\varphi_2(s)=(\mathcal{B}g)(-s)$ is a proper closed function. Concretely, it follows that
\begin{equation*}
  \varphi_2(s)=\gamma \Vert \pi(-[\mathbf{0}_{mn\times N^*}~I_{mn}]s)\Vert_0+\delta_{A(s)}(s)
\end{equation*}
with $A(s)=\{s\colon [I_{N^*}~\mathbf{0}_{N^*\times mn}]s=0\}$.
\item   Suppose the parameters $(A, B_2)$ of LTI system \eqref{system} are known. Then, $A$ is a Hurwitz matrix iff $\mathrm{dom}(\varphi_1) = \mathbb{R}^{N^*+mn}$.
\end{enumerate}
\end{lemma}

\begin{proof}
  See Appendix for the proof.
\end{proof}

By Lemma \ref{subgrad_varphi1}, problem \eqref{admm-standard} can be rewritten as
\begin{equation}\label{subgrad-flow2}
  \min_{s_2\in\mathbb{R}^{mn}}~\zeta(s_2)\doteq\underbrace{(\mathcal{A}f)(Us_2)}_{\doteq \widehat{\varphi}_1(s_2)}+\underbrace{\gamma \Vert \pi(s_2)\Vert_0}_{\doteq \widehat{\varphi}_2(s_2)}
\end{equation}
with $ U=\left(\begin{smallmatrix}
    \mathbf{0}_{N^*\times mn} \\
      I_{mn}
\end{smallmatrix}\right)$. We next characterize the variational structure of problem \eqref{subgrad-flow2}. 
In particular, every stationary point admits a strong optimality property: it is not only a local minimizer of \eqref{subgrad-flow2}, but also a global minimizer of an auxiliary equality-constrained convex problem.

\begin{theorem}\label{local-minimal}
Under Assumption \ref{ass1} and \ref{ass2}, let $\bar{s}_2\in \mathrm{ri}(\mathrm{dom}(\widehat{\varphi}_1))$ be a stationary point of problem \eqref{subgrad-flow2}, i.e., $ 0\in \partial \zeta(\bar{s}_2)$.
Then, $\bar{s}_2$ is a local minimizer. Moreover, there exists $\w_{\bar{s}_2}\subseteq [st]$, such that $\bar{s}_2$ is a global minimizer of 
\begin{equation}\label{auxi1}
  \min_{s_2\in\mathbb{R}^{mn}}~\widehat{\varphi}_1(s_2),~\text{s.t.}~[\pi(s_2)]_i=0,\forall i\in (\w_{\bar{s}_2})^c. \tag{$Q_{\w_{\bar{s}_2}}$}
\end{equation}
\end{theorem}

\begin{proof}
See Appendix for the proof.
\end{proof}

\subsection{Convergence Analysis}\label{chap-smooth}

In this subsection, we will study the convergence of the DR splitting algorithm for solving problem \eqref{subgrad-flow2}. Consider the  DR splitting iterative scheme  \cite{themelis2020douglas}:
\begin{align}
& u_{k}\in\mathbf{prox}_{\eta\widehat{\varphi}_1}(s_{2;k}), \notag\\
& v_{k}\in\mathbf{prox}_{\eta\widehat{\varphi}_2}(2u_{k}-s_{2;k}),\label{DR}\\
& s_{2;k+1}=s_{2;k}+\xi (v_{k}-u_{k})\notag
\end{align}
with $\eta,\xi>0$. Remarkably, applying the DR splitting algorithm \eqref{DR} to problem \eqref{subgrad-flow2} is equivalent to using ADMM algorithm:
\begin{align}\label{ADMM-intro}
  &u_{k+1}=\lambda_k-\eta^{-1}(1-\xi)(\mathcal{A}\widetilde{W}_k+\mathcal{B}\widetilde{P}_k),\notag\\
  &\widetilde{W}_{k+1}\in\argmin_{\widetilde{W}}
~ \mathcal{L}_{\eta^{-1}}(\widetilde{W},\widetilde{P}_k,u_{k+1}), \\
  &\lambda_{k+1}=u_{k+1}+\eta^{-1} (\mathcal{A}\widetilde{W}_{k+1}+\mathcal{B}\widetilde{P}_k),\notag\\
  &\widetilde{P}_{k+1}\in\argmin_{\widetilde{P}}~\mathcal{L}_{\eta^{-1}}(\widetilde{W}_{k+1},\widetilde{P},\lambda_{k+1})\notag
\end{align}
 for problem \eqref{sparse_feedback_6}; see Theorem 5.5 of \cite{themelis2020douglas} for the details. Here, the augmented Lagrangian is given by
$
  \mathcal{L}_{\eta^{-1}}(\widetilde{W},\widetilde{P},\lambda)\doteq  f(\widetilde{W})+g(\widetilde{P})+\langle \lambda, \mathcal{A}\widetilde{W}+\mathcal{B}\widetilde{P}\rangle
  +\frac{1}{2\eta}\Vert  \mathcal{A}\widetilde{W}+\mathcal{B}\widetilde{P}\Vert^2
$. By virtue of this equivalence, all results established in this subsection for the DR splitting algorithm \eqref{DR} naturally apply to the ADMM algorithm \eqref{ADMM-intro} for solving problem \eqref{sparse_feedback_6}; for brevity, we shall not explicitly restate them hereafter.

To provide the convergence analysis of the DR splitting algorithm \eqref{DR}, it is crucial to illustrate the smoothness of $\widehat{\varphi}_1$.
Concretely, $\widehat{\varphi}_1(s_2)$ can be characterized as the optimal value of the following SDP problem:
  \begin{align}
  \min_W~&\langle R,W\rangle \notag\\
  \text{s.t.}~ &\mathcal{L}_{0}(W)= Us_2,~~\mathcal{L}_i(W)=\Psi_i,~i\in[M],\label{sdp1}\\
  & { W\times \left(\otimes_{i=1}^{M}\Psi_i\right)} \in \mathcal{K}_+\doteq\mathbb{S}_+^p\times \left(\otimes_{i=1}^{M}\mathbb{S}_+^n\right)\notag
  \end{align}
with linear operator $\mathcal{L}_i(W)\doteq-V_2(F_iW+WF_i^\top+Q)V_2^\top$ for $i\in[M]$.
Here, equality constraint $\mathcal{L}_{0}(W)= Us_2$ is the rewriting of $\mathcal{A}\widetilde{W}=Us_2$.
Problem \eqref{sdp1} can be converted to standard form:
  \begin{align*}
  \min_{\mathbf{X}\in \mathcal{K}_+}~&\langle \mathbf{R},\mathbf{X} \rangle\\
   \mathrm{s.t.}~&\underbrace{\begin{bmatrix}
                    \mathcal{L}_0 &  &  \ & \\
                    -\mathcal{L}_1 & \mathcal{I} &  \ &  \\
                    \vdots & \vdots &\ \ddots & \\
                    -\mathcal{L}_M& 0  & \cdots & \mathcal{I}
                   \end{bmatrix}}_{\doteq \mathbf{L}}
                  \mathbf{X}=\underbrace{
  \begin{pmatrix}
    Us_2 \\
    \mathbf{0}_{n\times n} \\
    \vdots\\
    \mathbf{0}_{n\times n}
  \end{pmatrix}}_{\doteq \mathbf{b}}
  \end{align*}
  with $\mathbf{R}\doteq(R,\mathbf{0}_{n\times n},\dots,\mathbf{0}_{n\times n})^\top$ and $\mathbf{X}\doteq(W,\Psi_1,\dots,\Psi_M)^\top$,
which exhibits the following dual form:
\begin{equation*}
  \begin{aligned}
  \max_{\mathbf{y},\mathbf{S}}~ \mathbf{b}^\top \mathbf{y}~~\text{s.t.}~~& \mathbf{L}^*\mathbf{y}+\mathbf{S}=\mathbf{R},~~\mathbf{S} \in\mathcal{K}_+,
  \end{aligned}
\end{equation*}
where $\mathbf{y}\doteq (y_0^\top,y_1^\top,\dots,y_M^\top)^\top$ is the multiplier w.r.t. equality constraint $\mathbf{L}\mathbf{X}=\mathbf{b}$ and  $\mathbf{L}^*$ is the adjoint of $\mathbf{L}$. Solving problem \eqref{sdp1} is equivalent to solving equation:
\begin{equation}\label{kkt-system}
  F(\mathbf{X},\mathbf{y},\mathbf{S})\doteq\begin{bmatrix}
         \mathbf{R}-\mathbf{L}^*\mathbf{y}-\mathbf{S} \\
         \mathbf{L}\mathbf{X}-\mathbf{b} \\
         \mathbf{S}-\mathbf{Proj}_{\mathcal{K}_+}[\mathbf{S}-\mathbf{X}]
       \end{bmatrix}=0;
\end{equation}
see B. Eaves (1971) \cite{eaves1971basic}. By Theorem 11.50 of \cite{rockafellar-book} and Theorem 4.20 of \cite{ref5}, it follows that
\begin{equation*}
\begin{aligned}
  \partial \widehat{\varphi}_1(s_2)= &\left\{U^\top \bar{y}_0\colon (\bar{\mathbf{X}},\bar{\mathbf{\mathbf{y}}},\bar{\mathbf{S}}) \text{~solves equation}\right.\\ 
  &\left.\text{\eqref{kkt-system} with~}\bar{\mathbf{y}}\doteq(\bar{y}_0^\top,\dots,\bar{y}_M^\top)^\top\right\}.
\end{aligned}
\end{equation*}
Henceforth, the study of the smoothness of $\widehat{\varphi}_1(s_2)$ is reduced to the conditions under which equation \eqref{kkt-system} admits a unique solution. The following lemma provides the sufficient conditions under which $F(\mathbf{X},\mathbf{y},\mathbf{S})$ is a locally Lipschitz homeomorphism.

\begin{lemma}[Theorem 18 of \cite{chan2008constraint}]\label{chan}
Denoting $\mathcal{K}\doteq \mathbb{S}^p \times \left(\otimes_{i=1}^M \mathbb{S}^n\right)$ and $\mathbf{N}\doteq\mathrm{dim}(\mathbf{b})=N^*+mn+Mn^2$, let $(\bar{\mathbf{X}},\bar{\mathbf{\mathbf{y}}},\bar{\mathbf{S}})\in \mathcal{K}\times \R^\mathbf{N}\times \mathcal{K}$ satisfy equation \eqref{kkt-system}. Then, the following are equivalent. 
\begin{enumerate}
  \item The following nondegeneracy conditions hold:
  \begin{equation}\label{non}
    \begin{aligned}
    & \mathbf{L} \left( \mathrm{span} \left(\mathcal{T}_{\mathcal{K}_+}(\bar{\mathbf{X}})\right) \right)= \mathbb{R}^\mathbf{N},\\
    & \mathbf{L}^* \left(\mathbb{R}^\mathbf{N}\right) + \mathrm{span} \left(\mathcal{T}_{\mathcal{K}_+}(\bar{\mathbf{S}})\right) = \mathcal{K},
    \end{aligned}
  \end{equation}
      where $\mathcal{T}_{\mathcal{K}_+}(\bar{\mathbf{X}})$ denotes the tangent cone of $\mathcal{K}_+$ at $\bar{\mathbf{X}}$ and $\mathbf{span}\left(\mathcal{T}_{\mathcal{K}_+}(\bar{\mathbf{X}})\right)$ is the largest linear space in $\mathcal{T}_{\mathcal{K}_+}(\bar{\mathbf{X}})$.
  \item $F$ is a locally Lipschitz homeomorphism near $(\bar{\mathbf{X}},\bar{\mathbf{\mathbf{y}}},\bar{\mathbf{S}})$.
\end{enumerate}

\end{lemma}

By Lemma \ref{chan}, under the nondegeneracy conditions \eqref{non}, $\widehat{\varphi}_1$ is locally $L$-smooth. Thus, we have a clear structural understanding of $\widehat{\varphi}_1$: it consists of multiple locally $L$-smooth branches $\{\mathcal{S}_i\}\subseteq \mathbb{R}^{mn}$, each associated with a nondegenerate solution of equation \eqref{kkt-system}. At the boundaries between these branches, the nondegeneracy conditions \eqref{non} fail, leading to a  breakdown of smoothness, i.e., $\widehat{\varphi}_1$ exhibits a nonsmooth transition between smooth regimes. Empirically, DR splitting algorithm \eqref{DR} typically enters and remains within one of these smooth branches after finite iterations (often referred to as the ``active manifold identification'' property; see Section 3.2 of \cite{bai2026avoiding}). Motivated by this observation, throughout this subsection, we focus on the convergence behavior after the iterates have entered a specific smooth regime, as formalized below.

\begin{assumption}\label{lsmooth}
Assume that the sequences $\{u_k\}_{k\geq 0}$ and $\{s_{2;k}\}_{k\geq 0}$ generated by the DR splitting algorithm \eqref{DR} is contained in a closed convex region $\mathcal{S}^* \subseteq \mathrm{dom}(\widehat{\varphi}_1)$ s.t. $\widehat{\varphi}_1$ is $L_{\widehat{\varphi}_1}$-smooth on $\mathcal{S}^*$.
\end{assumption}

\begin{theorem}\label{smoothDR}
Under Assumption \ref{ass1}-\ref{lsmooth}, suppose that  the parameters satisfy $\eta<\min\left\{\frac{2-\xi}{2L_{\widehat{\varphi}_1}},\frac{1}{L_{\widehat{\varphi}_1}}\right\}$ and $\xi\in(0,2)$.
Let $\{(u_k, v_k, s_{2;k})\}_{k \geq 0}$ be the sequence generated by the DR splitting algorithm \eqref{DR}. Then, the sequences $\{u_k\}_{k\geq 0}$ and $\{v_k\}_{k\geq 0}$ converge to the same limiting point $u^\prime$, which satisfies the first-order optimality condition $0\in {\partial}\zeta(u^\prime)$.
\end{theorem}

\begin{proof}
See Appendix for the proof. 
\end{proof}

By Theorem \ref{local-minimal} and Theorem \ref{smoothDR}, there exists an index set $\w_{u^\prime}\subseteq [st]$, depending on $u^\prime$, such that $u^\prime$ is a global minimizer of
\begin{equation}\label{pf33}
  \begin{aligned}
  \min_{s_2\in\mathcal{S}^*}~&\widehat{\varphi}_1(s_2),~~\text{s.t.}~& [\pi(s_2)]_i=0,~ i\in\w_{u^\prime}^c.
  \end{aligned}
  \tag{$Q_{\w_{u^\prime}}$}
\end{equation}
Problem \eqref{pf33} is further equivalent to 
\begin{align*}
  \min_{W}~~&\langle R,W\rangle\\
  {\rm s.t.}~~&W\in\mathbb{S}^p_+,~~\Psi_i\in\mathbb{S}_+^{n},~~\forall i\in[M],\\
  &-V_{j1}WV_{j2}=0,~~\forall j\in[N],\\
  & \mathbf{gsupp}(W_2^\top)\subseteq \mathcal{S}_{\w_{u^\prime}}\subseteq [s]\times[t].
\end{align*}
Here, $\mathcal{S}_{\w_{u^\prime}}$ is uniquely determined by the index set $\w_{u^\prime}$ through the canonical identification between $[st]$ and $[s]\times[t]$, and $\mathbf{gsupp}(W_2^\top)$ denotes the group support of $W_2^\top$. 
The above problem admits a clear physical interpretation: it corresponds to a DFT-LQ problem with admissible communication topology $\mathcal{S}_{\w_{u^\prime}}$.
This observation reveals one of the main advantages of directly studying the $\ell_0$-induced SF-LQ problem. 
Indeed, the SF-LQ and DFT-LQ problems become intrinsically connected through the support identification mechanism induced by the $\ell_0$-penalty. 
More precisely, limit point $u^\prime$ of the iteration sequence generated by the DR splitting algorithm \eqref{DR} automatically determines an admissible communication topology $\mathcal{S}_{\w_{u^\prime}}$, and $u^\prime$ is exactly a global minimizer of the associated DFT-LQ problem under this fixed communication topology. 
Consequently, the asymptotic behavior of the algorithm admits a natural control-theoretic interpretation: The algorithm not only identifies a group-sparse communication architecture, but also simultaneously finds the global minimizer of the DFT-LQ problem associated with such architecture.

\section{Projected Subgradient Descent}\label{chap-nonsmooth}

In the previous section, the DR splitting algorithm was analyzed under Assumption \ref{lsmooth}, which guarantees that the iterates stay inside a fixed $L_{\widehat{\varphi}_1}$-smooth manifold $\mathcal{S}^*$.  
Indeed, such a smooth manifold is typically not identified at the initial stage; rather, the iterates only eventually enter and remain within this smooth branch after stabilizing (again, referred to as the ``active manifold identification'' property).
Therefore, a practical strategy is to first apply a globally stable method to steer the iterates toward a stable $L_{\widehat{\varphi}_1}$-smooth branch, and  then employ the DR splitting algorithm to benefit from its support-identification mechanism. 
In this section, we propose a Projected Subgradient Descent (PSGD) algorithm, which enjoys global convergence guarantees and can serve as a warm-start procedure for DR splitting algorithm \eqref{DR}. 
Specifically, we will investigate problem \eqref{subgrad-flow2} without Assumption \ref{lsmooth}, and approximate $\Vert\pi(s_2)\Vert_0$ by 
\begin{equation*}
  \vartheta_{\alpha,h}(s_2)\doteq h(\widetilde{\pi}(s_2))-\underbrace{\inf_{u\in\mathbb{R}^{st}}\left\{h(u)+\frac{1}{2\alpha}\Vert u-\widetilde{\pi}(s_2)\Vert\right\}}_{\doteq \mathbf{env}_\alpha(h)(\widetilde{\pi}(s_2))}
\end{equation*}
with closed convex $h$ and $\alpha>0$. By Theorem 4.1 of \cite{shen2019structured}, if $h$ is a sparsity promoting function (see Definition 3.1 of \cite{shen2019structured}), then $\vartheta_{\alpha,h}$ is also a sparsity promoting function. 
Henceforth, the following relaxation optimization problem is of concern:
\begin{equation}\label{subgrad-flow3}
  \min_{s_2\in\mathbb{R}^{mn}}~{\zeta}_{\alpha, h}(s_2)\doteq\widehat{\varphi}_1(s_2)+\frac{\gamma}{c_{h,\alpha}}\vartheta_{\alpha,h}(s_2)
\end{equation}
with constant $c_{h,\alpha}>0$. 

\begin{theorem}\label{thm_app}
  Under Assumption \ref{ass1} and \ref{ass2}, let $\alpha_k\downarrow 0$, then it holds that 
  $\inf\limits_{s_2\in\R^{mn}} \zeta_{\alpha_k,h}(s_2) \to \inf\limits_{s_2\in\R^{mn}} \zeta(s_2)$.
 In addition, for any choice of $\epsilon_k\downarrow 0$ and $s_2^{(k)}\in\epsilon_k\text{-}{\argmin\limits_{s_2\in\R^{mn}}}~\zeta_{\alpha_k,h}(s_2)$,
 the sequence $\{s_2^{(k)}\}_{k\geq 1}$ is bounded and such that all its cluster points belong to $\argmin\limits_{s_2\in\R^{mn}}\zeta(s_2)$.
\end{theorem}

\begin{proof}
See Appendix for the proof.
\end{proof}

Given $\epsilon>0$ sufficiently small and defining
\begin{align*}
  &\widehat{\Xi}_\epsilon \doteq
  \left\{ \mathrm{vec}(W) \left|
  \begin{aligned}
  &W \succeq \epsilon I,\\
  &\Psi_i \succeq \epsilon I,~i\in [M]
  \end{aligned}
  \right.
  \right\}\subseteq \mathrm{ri}(\mathrm{dom}(f)),\\
  &\Xi_\epsilon \doteq   
\left\{
    \begin{aligned}
    &U^{-1}(\mathcal{A}(\widehat{\Xi}_\epsilon)),~\text{if}~A~\text{is not Hurwitz},\\
    &\mathbb{R}^{mn},~~~~~~~~~~\text{if}~A\text{~is Hurwitz},
    \end{aligned}
\right.
\end{align*}
it follows that $\Xi_\epsilon \subseteq \mathrm{ri}(\mathrm{dom}(\widehat{\varphi}_1)).$
Now, we study the following modification problem of \eqref{subgrad-flow3}
\begin{equation}\label{subgrad-flow4}
  \min_{s_2\in[-C,C]^{mn}\cap \Xi_\epsilon}~{\zeta}_{\alpha, h}(s_2)\doteq\widehat{\varphi}_1(s_2)+\frac{\gamma}{c_{h,\alpha}}\vartheta_{\alpha,h}(s_2)
\end{equation}
with sufficiently large $C>0$ by the PSGD algorithm:
\begin{equation}\label{subgrad-alg}
    \text{Select~}g_k\in   \partial {\zeta}_{\alpha, h}(s_{2;k}),~s_{2;k+1}^\prime\in  \mathbf{Proj}_{\Xi_\epsilon}(s_{2;k}-\alpha_k g_k),
\end{equation}
 where step-size sequence $\{\alpha_k\}_{k\geq 1}$ satisfies $\alpha_k\geq0$, $\sum_{k=1}^{\infty} \alpha_k=\infty$ and $\sum_{k=1}^{\infty} \alpha_k^2<\infty$.


\begin{theorem}\label{thm7}
  Under Assumption \ref{ass1} and \ref{ass2}, let $\{s_{2;k}\}_{k\geq 1}$ be the iterates produced by PSGD algorithm \eqref{subgrad-alg}. Then for all cluster point $s_2^*$ of $\{s_{2;k}\}_{k\geq 1}$, it holds that
$
    0 \in \partial \zeta_{\alpha,h}(s_2^*)+ \mathcal{N}_{\Xi_\epsilon}(s_2^*)
$,
and the function value sequence $\{\zeta_{\alpha,h}(s_{2;k})\}_{k\geq 1}$ converges.
\end{theorem}


\begin{proof}
See Appendix for the proof.
\end{proof}

\section{Numerical Examples}\label{section:numercial}

In this section, we provide several numerical examples to demonstrate the numerical effectiveness of the proposed algorithm. Specifically, we use ADMM algorithm \eqref{ADMM-intro} to solve optimization problem \eqref{sparse_feedback_6} (which is equivalent to utilizing DR splitting \eqref{DR} for problem \eqref{subgrad-flow2}), and the noise $w$ of \eqref{system} is characterized by an impulse disturbance vector.

\begin{example}\label{exp1}
Consider $x=[x_1,x_2,x_3]^\top$ and the system parameters $(A|B_1|B_2|C|D)$ are given by
\begin{equation*}
  \begin{pmatrix}
\begin{array}{ccc|c|cc|ccc|cc}
0 & 1 & 0 &     & 0.9315 & 0.7939 & 1 & 0 & 0 & 0 & 0 \\
0 & 0 & 1 & I_3 & 0.9722 & 0.1061 & 0 & 0 & 0 & 1 & 0 \\
0 & 0 & 0 &     & 0.5317 & 0.7750 & 0 & 0 & 0 & 0 & 1
\end{array}
\end{pmatrix}.
\end{equation*}
 We aim to solve the corresponding SF-LQ problem \eqref{sparse_feedback_6} with the block structure $K = \mathbf{block}_{2,2}(K_{11}, K_{12}, K_{21}, K_{22})$, where $K_{11}\in\mathbb{R}^{1\times 2}$, $K_{12}\in\mathbb{R}$, $K_{21}\in\mathbb{R}^{1\times 2}$ and $K_{22}\in\mathbb{R}$.
\end{example}

\textbf{Solution.}  
First, we study the convergence behaviour of ADMM algorithm \eqref{ADMM-intro} for SF-LQ problem \eqref{sparse_feedback_6} and verify the correctness of our theoretical findings. Denoting $\{(u_k, \widetilde{W}_k, \lambda_k, \widetilde{P}_k)\}_{k\geq 0}$ as the iterative sequence of ADMM, we select the initial point $(u_0, \widetilde{W}_0, \lambda_0, \widetilde{P}_0)=(\mathbf{0}_{8\times 1}, \mathbf{0}_{25\times 1}, \mathbf{0}_{8\times 1}, \mathbf{0}_{6\times 1})$, $\eta=1/30$ and $\xi=1$. We first verify the feasibility of the computed solution (see Fig. \ref{ADMM-l0-new1}). The resulting matrix $W_1^*$ is given by 
\begin{equation*}
  W_1^*=\begin{bmatrix}
 1.024 & -1.126 &  0.000 & -0.061 &  0.000 \\
-1.126 &  2.587 &  0.000 &  1.366 &  0.000 \\
 0.000 &  0.000 &  1.670 &  0.000 &  0.924 \\
-0.061 &  1.366 &  0.000 &  1.256 &  0.000 \\
 0.000 &  0.000 &  0.924 &  0.000 &  0.511 
\end{bmatrix},
\end{equation*}
which corresponds exactly to the global minimizer of the following DFT-LQ problem:
\begin{align*}
  \min_{W,P}~~&\langle R,W\rangle\\
  {\rm s.t.}~~&W\in\mathbb{S}^p_+,\\
  &-V_2(FW+WF^\top+Q)V_2^\top\in\mathbb{S}_+^{n},\\
  &W_{13}=W_{23}=0,\\
  &W_{43}=W_{51}=W_{52}=0.
\end{align*}
The associated feedback gain is obtained as $K_1^*=\left(\begin{smallmatrix} 1.000 & 0.967 & 0.000 \\ 0.000 & 0.000 & 0.573 \end{smallmatrix}\right)$ with $J(K_1^*)=2.792$. 
Alternatively, if the initial point of ADMM is selected as $(u_0, \widetilde{W}_0, \lambda_0, \widetilde{P}_0)=(\mathbf{0}_{8\times 1}, \mathbf{4}_{25\times 1}, \mathbf{0}_{8\times 1}, \mathbf{0}_{6\times 1})$, the algorithm yields a different solution:
\begin{equation*}
  W_2^*=\begin{bmatrix}
 0.889 & -1.093 &  0.000 & -0.037 & -0.208 \\
-1.093 &  2.653 &  0.000 &  1.100 &  0.842 \\
 0.000 &  0.000 &  1.953 &  0.000 &  0.837 \\
-0.037 &  1.100 &  0.000 &  0.851 &  0.481 \\
-0.208 &  0.842 &  0.837 &  0.481 &  0.670 
\end{bmatrix}.
\end{equation*}
Notably, $W_2^*$ coincides with the global minimizer of the following DFT-LQ problem:
\begin{align*}
  \min_{W,P}~~&\langle R,W\rangle\\
  {\rm s.t.}~~&W\in\mathbb{S}^p_+,\\
  &-V_2(FW+WF^\top+Q)V_2^\top\in\mathbb{S}_+^{n},\\
  &W_{13}=W_{23}=0,\\
  &W_{43}=0.
\end{align*} 
In this case, the resulting feedback gain is $K_2^*=\left(\begin{smallmatrix} 0.949 & 0.805 & 0.000 \\ 0.315 & 0.448 & 0.429 \end{smallmatrix}\right)$ with $J(K_2^*)=2.410$.
It is worth emphasizing that both aforementioned initial points are sufficiently close to certain stationary points, thereby satisfying Assumption \ref{lsmooth}. Furthermore, numerical experiments validate the correctness of Theorem \ref{local-minimal}, confirming that the stationary point of the $\ell_0$-induced SF-LQ problem and the global minimizer of the DFT-LQ problem exhibit a one-to-one correspondence.

Now, we aim to investigate the convergence behavior of ADMM when Assumption \ref{lsmooth} fails to hold. To do so, we select a ``bad'' initial point $(u_0, \widetilde{W}_0, \lambda_0, \widetilde{P}_0)=(\mathbf{0}_{8\times 1}, \mathbf{50}_{25\times 1}, \mathbf{0}_{8\times 1}, \mathbf{0}_{6\times 1})$, $\eta=0.1$ and $\xi=0.5$. As before, we start by verifying the feasibility of the solution; see Fig. \ref{ADMM-l0-10}. We may observe that sequences $\{[W_k]_{13}\}_{k\geq 0}$ and $\{[W_k]_{23}\}_{k\geq 0}$ quickly approach $0$ and exhibit regular oscillations around it; however
the sequence $\{\Vert P_k-W_{2,k}^\top\Vert\}_{k\geq 0}$ exhibits a small deviation from the origin.
Interestingly, by selecting a smaller $\eta=0.01$, it holds that $0$ belongs to the cluster point of$\{\Vert P_k-W_{2,k}^\top\Vert\}_{k\geq 0}$, and the oscillation amplitudes of the sequences $\{[W_k]_{13}\}_{k\geq 0}$, $\{[W_k]_{23}\}_{k\geq 0}$ and $\{\Vert P_k-W_{2,k}^\top\Vert\}_{k\geq 0}$ are significantly reduced; see Fig. \ref{ADMM-l0-100}.
Furthermore, through selecting $\eta=1/300$, the sequences $\{[W_k]_{13}\}_{k\geq 0}$, $\{[W_k]_{23}\}_{k\geq 0}$ and $\{\Vert P_k-W_{2,k}^\top\Vert\}_{k\geq 0}$ converge to $0$; see Fig. \ref{ADMM-l0-300}. In summary, we find that when Assumption \ref{lsmooth} is not satisfied, the convergence behavior of the algorithm is highly sensitive to the choice of $\eta$. Specifically, a smaller $\eta$ exhibits better convergence, though at the expense of sacrificing the group-sparsity level of the resulting $\widetilde{P}$. 
Remarkably, even though full convergence is lost, by selecting a subsequence from the iterative sequence, we can still recover the group-sparse feedback gain $K$.
For instance, letting $\eta = 1/30$, we may choose a proper subsequence $\mathbb{K}\subseteq \mathbb{N}$ such that $\{W_k\}_{k\in\mathbb{K}}$ converges to
\begin{equation*}
	W_3^* = \begin{bmatrix}
		6.346 & -2.709 & 0.000 & 0.200 & 0.000\\
		-2.709 & 1.172 & 0.000 & 1.115 & 0.000\\
		0.000 & 0.000 & 0.759 & 0.000 & 1.026\\
		0.200 & 1.115 & 0.000 & 7.861 & 3.460\\
		0.000 & 0.000 & 1.026 & 3.460 & 2.615
	\end{bmatrix}
\end{equation*}
and the feedback gain is given by $K_3^*=\left(\begin{smallmatrix}
 		32.934 & 77.077  & 0       \\
 0       &  0        &  1.354
                                               \end{smallmatrix}\right)$
with $J(K_3^*)=6.346$. Hence, in this case, the iterative sequence of ADMM no longer exhibits sequential convergence to a stationary point of problem \eqref{subgrad-flow2}. This underscores that the convergence behavior of ADMM under the violation of Assumption \ref{lsmooth} falls outside the scope covered by our theoretical framework, which can serve as a promising avenue for future research. Nevertheless, if we employ the PSGD algorithm as a warm start from such a ``bad'' initial point, and subsequently initialize ADMM with the resulting warm-started point, we can observe that ADMM successfully converges to $W_2^*$.  This demonstrates the effectiveness of the warm-start strategy  proposed in Section \ref{chap-nonsmooth}.

Now, we solve SF-LQ problem \eqref{sparse_feedback_6} by existing methods.

\textbf{(1) Combinatorial Optimization: MISDP Approach.}

Problem \eqref{sparse_feedback_6} can be reformulated into the following MISDP (Mixed-Integer SDP) problem
  \begin{align}\label{misdp}
\min_{W\in\mathbb{S}^p,z\in\{0,1\}^{s\times t}}&\langle R,W\rangle+\gamma \sum_{i\in [s],j\in[t]}z_{ij}\notag\\
    {\rm s.t.}~~&W\in\mathbb{S}^p_+,\notag\\
    &\Psi_i\in\mathbb{S}_+^{n},~~\forall i\in[M],\\
    &-V_{j1}WV_{j2}=0,~~\forall j\in[N],\notag\\
    &\Vert W_{2,ij}\Vert_{\infty}\leq Mz_{ij},~~\forall i\in [s],j\in[t]\notag
  \end{align}
with $M=10^4$. 
We solve problem \eqref{misdp} by the Pajarito solver (see \url{github.com/JuliaOpt/Pajarito.jl}) in Julia v1.11. When $\gamma=10$, the Pajarito solver successfully finds the global minimizer of problem \eqref{misdp}, which coincides with $W_1^*$; whereas, for $\gamma=0.1$, the resulting global minimizer coincides with $W_2^*$. This once again demonstrates the efficacy of DR splitting \eqref{DR} (equivalently, ADMM \eqref{ADMM-intro}), showing that the points found by our method coincide with the global optimal solutions delivered by the Pajarito solver. 
However, the Pajarito solver relies on branch-and-bound and cutting-plane methods, which inherently suffer from the curse of dimensionality and become computationally intractable for large-scale problems. In contrast, the proposed algorithm scales effectively and can efficiently solve problems of much larger dimensions; see Example \ref{exp2new}.

\textbf{(2) Fully Convex Relaxation.} 

Feng and Ni (2024) \cite{feng2024optimization} studied the SF-LQ problem by the following convex relaxation problem
\begin{equation}\label{exp1-convex}
\begin{aligned}
\min_{\widetilde{W},\widetilde{P}}~~ &f(\widetilde{W})+\gamma\Vert \mathrm{vec}^{-1}(\widetilde{P}) \Vert_{s,t;1}\\
{\rm s.t.}~~&\mathcal{A}\widetilde{W}+\mathcal{B}\widetilde{P}=0
\end{aligned}
\end{equation}
with group $\ell_1$-norm $\Vert\cdot \Vert_{s,t;1}$. Remarkably, problem \eqref{exp1-convex} belongs to the class of convex optimization problems; hence many off-the-shelf algorithms can be employed to solve it. When $\gamma=50$, the feedback gain is given by $K_4^*=\left(\begin{smallmatrix}
1.000 & 2.619 & 0.000 \\
0.000 & 0.000 & 1.600
             \end{smallmatrix}\right)$
with LQ cost $J(K_4^*)=7.307$, while letting $\gamma = 200$, we may obtain the feedback gain
$K_5^*=\left(\begin{smallmatrix}
1.000 & 3.817 & 0.000 \\
0.000 & 0.000 & 1.765
             \end{smallmatrix}\right)$
with LQ cost $J(K_5^*)=12.310$. We observe that the feedback gains obtained by this convex relaxation approach incur a significantly large LQ cost. Nevertheless, this approach is of independent interest, as it can also serve as an effective warm-start technique for the ADMM algorithm \eqref{ADMM-intro} when solving problem \eqref{sparse_feedback_6}. Specifically, given initial point  $(u_0, \widetilde{W}_0, \lambda_0, \widetilde{P}_0)=(\mathbf{0}_{8\times 1}, \mathbf{50}_{25\times 1}, \mathbf{0}_{8\times 1}, \mathbf{0}_{6\times 1})$  and using ADMM for solving problem \eqref{exp1-convex}, we may obtain 
\begin{equation*}
  W_{\rm init}^*=\begin{bmatrix}
 4.759 & -1.818 & 0.000 & -0.006 &  0.000 \\
-1.818 &  0.900 & 0.000 &  0.541 &  0.001 \\
0.000 & 0.000 &  0.454 &  0.001 &  0.727 \\
-0.006 &  0.541 &  0.000 &  1.410 &  0.027 \\
 0.000 &  0.000 &  0.727 &  0.027 &  1.165
\end{bmatrix}.
\end{equation*}
Then, we may select initial point $(u_0, \widetilde{W}_0, \lambda_0, \widetilde{P}_0)=(\mathbf{0}_{8\times 1}, \mathrm{vec}(W_{\rm init}^*), \mathbf{0}_{8\times 1}, \mathbf{0}_{6\times 1})$, $\eta=1/30$ and $\xi=1$, while use ADMM algorithm \eqref{ADMM-intro} for solving SF-LQ problem \eqref{sparse_feedback_6}. Ultimately, it can be observed that the resulting iterative sequence also converges to $W_1^*$.

\textbf{(3) Partial Convex Relaxation.} 

By replacing the group $\ell_0$-norm with its convex surrogate group $\ell_1$-norm, the SF-LQ problem can be relaxed to
\begin{equation}\label{exp1-opt}
         \begin{aligned}
           \min\limits_{K\in\mathcal{S}\subseteq\mathbb{R}^{m\times n}}~~&J(K)+\gamma \Vert K\Vert_{s,t;1}.\\
         \end{aligned}
\end{equation}
Lin \textit{et al.} (2013) \cite{kk4} rewrote problem \eqref{exp1-opt} into the following constrained form 
\begin{equation}\label{exp1-opt2}
\begin{aligned}
\min\limits_{\substack{K\in\mathcal{S}\subseteq\mathbb{R}^{m\times n}\\ G\in\mathbb{R}^{m\times n}}}~~&J(K)+\gamma \Vert G\Vert_{s,t;1}\\
{\rm s.t.}~~&K-G=0,
\end{aligned}
\end{equation}
and they solved the above problem by ADMM algorithm; see Section III of \cite{kk4} for details. Remarkably, due to the non-convexity of
$J(K)$ and $\mathcal{S}$, it remains an open problem whether the ADMM algorithm converges for solving \eqref{exp1-opt2}. Lin \textit{et al.} provided an open-source implementation (named LQRSP) of their ADMM-based algorithm; see \url{www.umn.edu/~mihailo/software/lqrsp/}. 
Specifically, by calling LQRSP with $\gamma=200$ and an all-zero initial point, we obtain feedback gain $K_6^*=\left(\begin{smallmatrix}
0.150 & 0.402 & 0.000 \\
0.000 & 0.000 & 0.000
             \end{smallmatrix}\right)$
with LQ cost $J(K_6^*)=17.064$.

In addition to the LQRSP solver proposed by Lin \textit{et al.}, Park \textit{et al.} (2020) \cite{park2020structured} directly studied the unconstrained problem \eqref{exp1-opt} from the perspective of policy iteration. Specifically, Park \textit{et al.} proposed the Structured Policy Iteration (S-PI) algorithm, which alternates between policy evaluation (solving Lyapunov equations) and policy improvement (proximal gradient descent:
\begin{align*}
  &F_k = K_k - \eta \nabla J(K)\\
  &K_{k+1}\in \underset{K}{\operatorname*{argmin}}~ \Vert K\Vert_{s,t;1}+\frac{1}{2\gamma \eta} \Vert K-F_k\Vert_F^2
\end{align*}
with step-size $\eta$). Although Park \textit{et al.} studied the discrete-time SF-LQ problem, we adapt their S-PI algorithm to the continuous-time SF-LQ setting of this paper and compare it with our ADMM algorithm \eqref{ADMM-intro}. Again, when $\gamma = 200$ and by S-PI algorithm with an all-zero initial point, we obtain the following feedback gain $K_7^*=\left(\begin{smallmatrix}
0.226 & 0.326 & 0.000 \\
0.000 & 0.000 & 0.000
             \end{smallmatrix}\right)$
with LQ cost $J(K_7^*)=18.300$. We observe that while both $K_6^*$ and $K_7^*$ exhibit relatively small Frobenius norms, their associated LQ costs are significantly large. This phenomenon arises because the $\ell_1$-penalty $\Vert K\Vert_{s,t;1}$ inherently introduces an over-shrinkage effect, forcing the algorithms to excessively penalize the magnitudes of the non-zero elements in order to achieve group-sparsity, which severely penalizes the control performance.

\textbf{(4) Penalty Function Approach.}

In \cite{feng-part1}, we employed a penalty function approach by introducing the following unconstrained optimization problem:
\begin{equation}\label{Pen_1}
\min_{\widetilde{W},\widetilde{P}}~~ F(\widetilde{W},\widetilde{P}):= f(\widetilde{W})+g(\widetilde{P})+H(\widetilde{W},\widetilde{P}),
\end{equation}
where $H(\widetilde{W},\widetilde{P})=\frac{\rho}{2}\Vert\mathcal{A}\widetilde{W}+\mathcal{B}\widetilde{P} \Vert^2$ represents the penalty term associated with the penalty parameter $\rho>0$. To solve problem \eqref{Pen_1}, we investigated the PALM algorithm parameterized by $(\sigma,\beta,\tau,\mu)$, whose iterative updates are governed by
  \begin{align*}
  &\widetilde{P}_{n+1}\in \mathbf{prox}_{\mu^{-1}g}(\widetilde{P}_n-\mu^{-1}\nabla_{\widetilde{P}}H(\widetilde{W}_n,\widetilde{P}_n)),\\
  &z_{n+1}\in\mathbf{prox}_{\beta^{-1}f}(\widetilde{W}_n+\beta^{-1}u_n),\\
  &\widetilde{W}_{n+1}=\widetilde{W}_n-\tau^{-1}(\nabla_{\widetilde{W}}H(\widetilde{W}_n,\widetilde{P}_{n+1})+u_n\\
  &~~~~~~~~~~+\beta(\widetilde{W}_n-z_{n+1})),\\
  &u_{n+1}=u_n+\sigma \beta (\widetilde{W}_{n+1}-z_{n+1}).
  \end{align*}
When $(\widetilde{P}_0,z_0,\widetilde{W}_0,u_0)=(\mathbf{0}_{6\times 1},\mathbf{0}_{25\times 1},\mathbf{0}_{25\times 1},\mathbf{0}_{25\times 1})$, PALM algorithm finds the feedback gain $K_8^*=\left(\begin{smallmatrix}
1.393 &  1.635 & 0.000 \\
0.000 &0.000 &  0.925
             \end{smallmatrix}\right)$
with LQ cost $J(K_8^*)=3.346$ (recalling $J(K_1^*)=2.792$ and $J(K_2^*)=2.410$); whereas, when $(\widetilde{P}_0,z_0,\widetilde{W}_0,u_0)=(\mathbf{50}_{6\times 1},\mathbf{50}_{25\times 1},\mathbf{0}_{25\times 1},\mathbf{50}_{25\times 1})$, the resulting feedback gain is given by $K_9^*=\left(\begin{smallmatrix}
1.219 & 1.188 & 0.000 \\
0.194 & 0.350 & 0.919
             \end{smallmatrix}\right)$
with LQ cost $J(K_9^*)=3.576$ (recalling $J(K_3^*)=6.346$). Consequently, both \cite{feng-part1} and the present work have independent research merit, revealing distinct operational behaviors of the two algorithms. On one hand, the ADMM algorithm for solving problem \eqref{sparse_feedback_6} (equivalently, DR splitting for solving problem \eqref{subgrad-flow2}) exhibits an elegant intrinsic mechanism: any stationary point it converges to inherently corresponds to a global minimizer of a certain DFT-LQ problem. On the other hand, the PALM algorithm for problem \eqref{Pen_1} demonstrates a more robust convergence profile. Notably, even when initialized from a highly adverse starting point, the PALM framework reliably handles the non-convex landscape and yields a group-sparse feedback gain that still incurs a relatively low LQ cost.

\begin{example}\label{exp2new}
Consider system \eqref{system} with the system matrices given by
\begin{align*}
  &A=\frac{1}{n}\mathrm{rand}(n,n)-I_n,~B_2=\mathbf{1}_{n\times m}+\frac{1}{2}\mathrm{rand}(n,m),\\
  &B_1=I_n,~C=\begin{bmatrix}
               I_n \\
               \mathbf{0}_{m\times n}
             \end{bmatrix},~D=\begin{bmatrix}
               \mathbf{0}_{n\times m} \\
               I_m
             \end{bmatrix}
\end{align*}
with $n=60$ and $m=30$, where $\mathrm{rand}(n,m)$ (with a fixed random seed of $42$) is a $n\times m$ matrix with every entry generated from the uniform distribution between $0$ and $1$. We aim to solve the corresponding SF-LQ problem \eqref{sparse_feedback_6} with the block structure $K=\mathbf{block}_{15,20}(K_{1,1},\dots,K_{15,20})$, where $K_{i,j}\in\mathbb{R}^{2\times 3}$ for all $i\in [15]$ and $j\in [20]$.
\end{example}

\textbf{Solution.} We study the convergence behavior of ADMM algorithm \eqref{ADMM-intro} for solving this large-scale SF-LQ problem. Again, we denote $\{(u_k, \widetilde{W}_k, \lambda_k, \widetilde{P}_k)\}_{k\geq 0}$ as the iterative sequence of ADMM and select an all-zero initial point, $\eta=0.02$ and $\xi=0.5$.  Let $\mathcal{N}$ denote the index set of the off-diagonal elements in $W_1$. It can be observed that both $\sum_{(i,j)\in\mathcal{N}}|[W_{k}]_{ij}|$ and $\|P_k-W_{2;k}\|_F$ decrease to a very small magnitude (on the order of $10^{-3}$) within $100$ iterations; see Fig. \ref{ADMM-high}. 
Moreover, at $k=150$, more than $70\%$ of the entries in $W_{2;k}$ become zero (where entries with an absolute value less than $10^{-4}$ are treated as numerical zeros).
Meanwhile, the closed-loop system is stabilized, as evidenced by the state response curves of the first ten states shown in Fig. \ref{response2}. This numerical example demonstrates that the proposed algorithm is scalable and well-suited for solving large-scale problems.

\begin{figure*}[htbp]
    \centering
    \begin{subfigure}[b]{0.45\textwidth}
        \includegraphics[width=\textwidth]{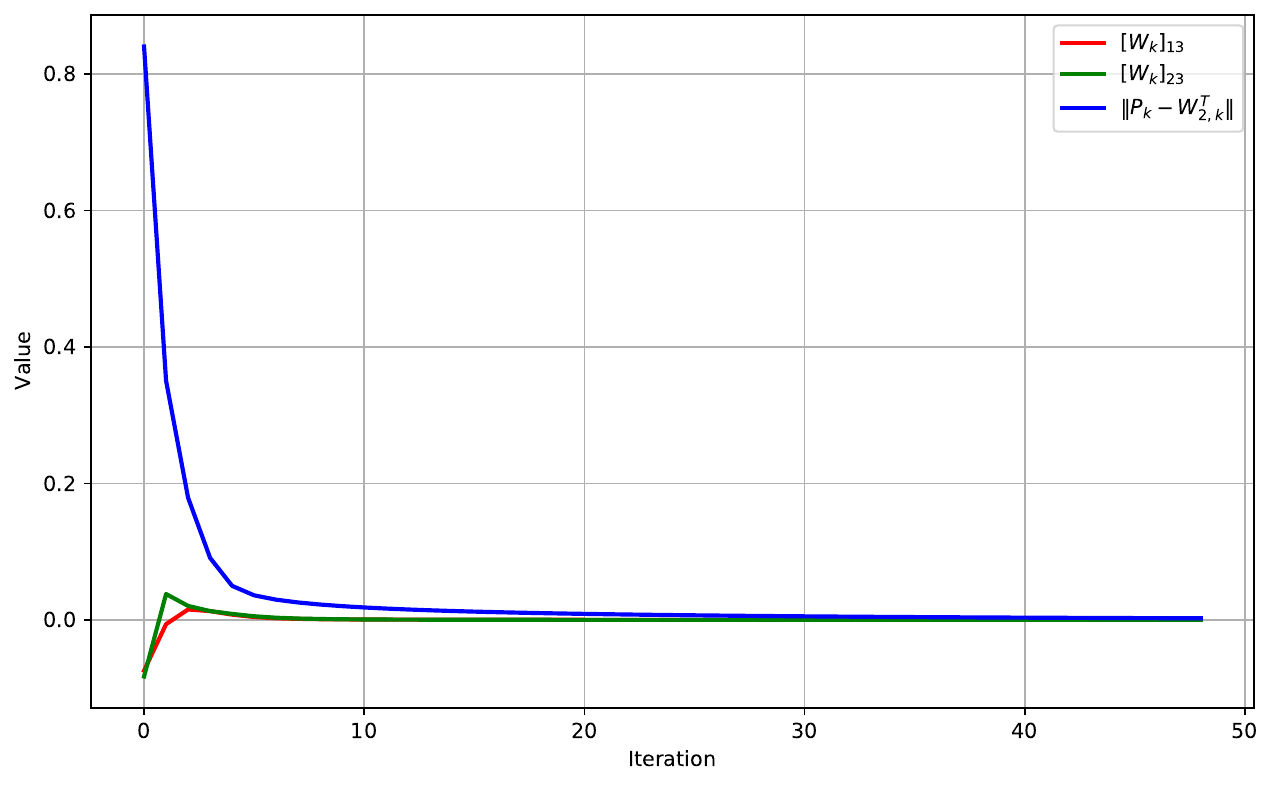}
        \caption{Feasibility of ADMM}
        \label{ADMM-l0-new1}
    \end{subfigure}
    \hfill
    \begin{subfigure}[b]{0.45\textwidth}
    \includegraphics[width=\textwidth]{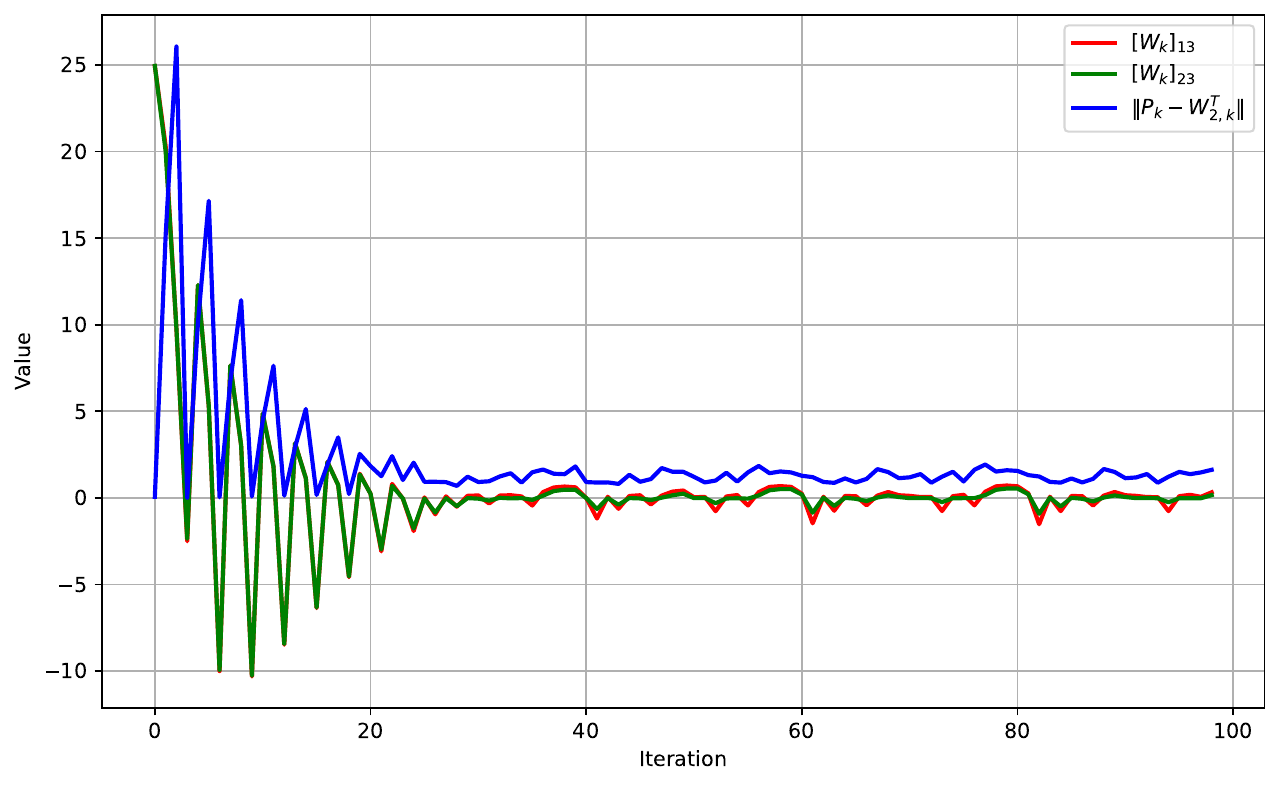}
    \caption{Non-Feasibility when $\eta=0.1$}
    \label{ADMM-l0-10}
    \end{subfigure}
    
    \vspace{0.3cm}

    \begin{subfigure}[b]{0.45\textwidth}
        \includegraphics[width=\textwidth]{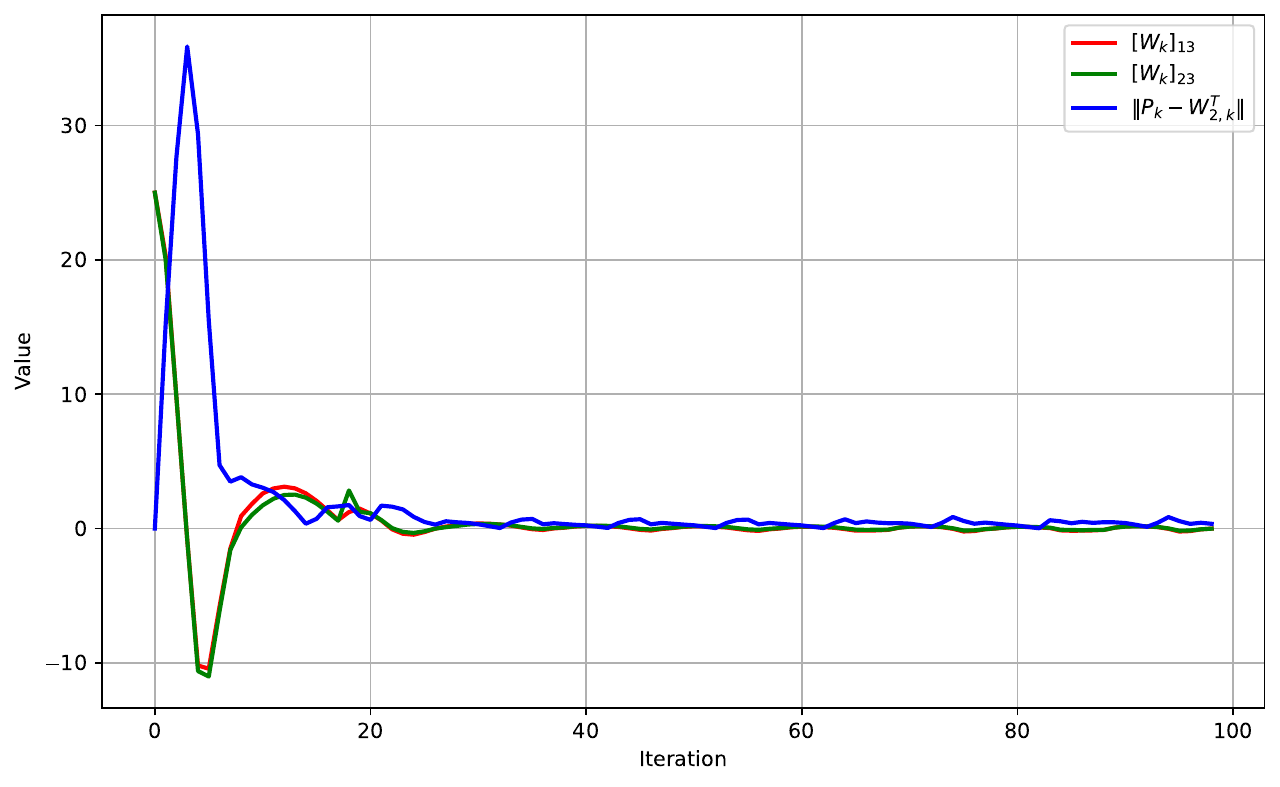}
        \caption{Subsequence Asymptotic Feasibility of $\Vert P_k-W_{2,k}^\top\Vert$ when $\eta=0.01$}
        \label{ADMM-l0-100}
    \end{subfigure}
    \hfill
     \begin{subfigure}[b]{0.45\textwidth}
    \includegraphics[width=\textwidth]{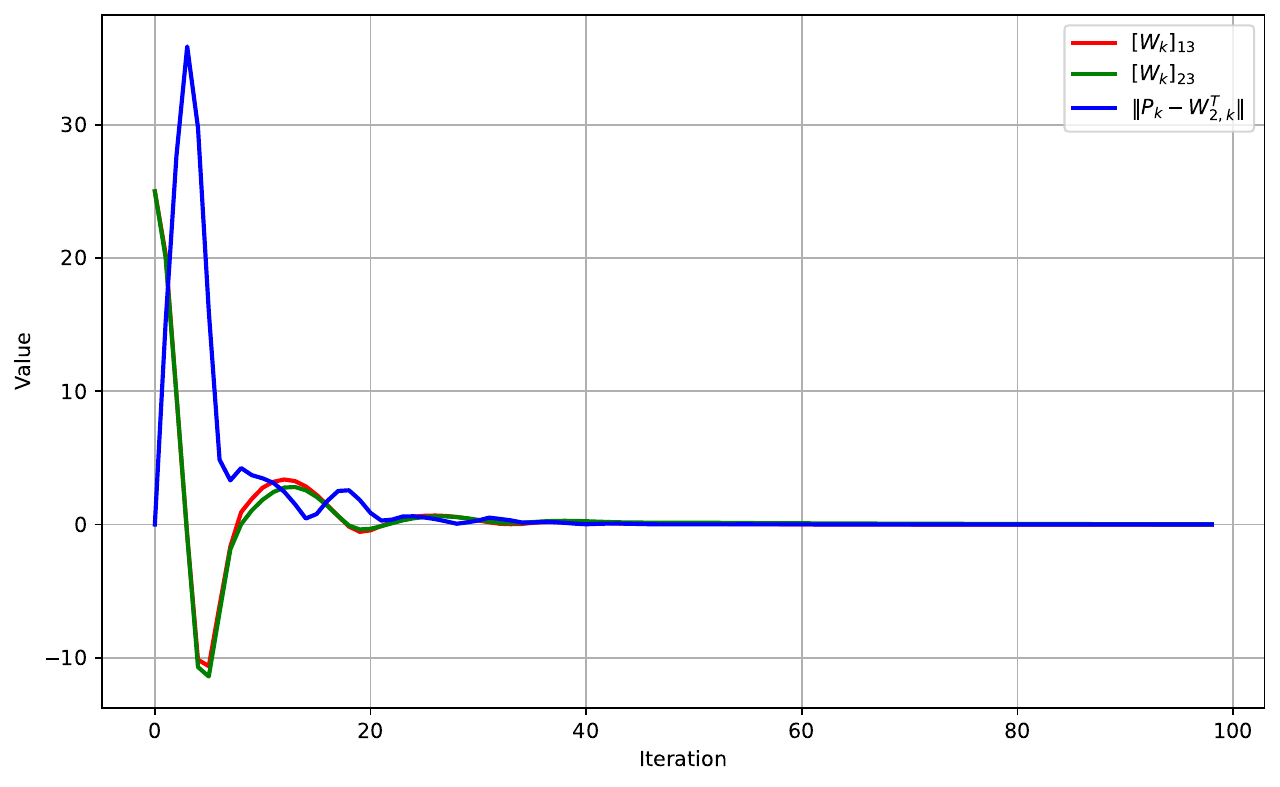}
    \caption{Asymptotic Feasibility when $\eta=1/300$}
    \label{ADMM-l0-300}
    \end{subfigure}
      \vspace{0.3cm}
     \begin{subfigure}[b]{0.45\textwidth}
        \includegraphics[width=\textwidth]{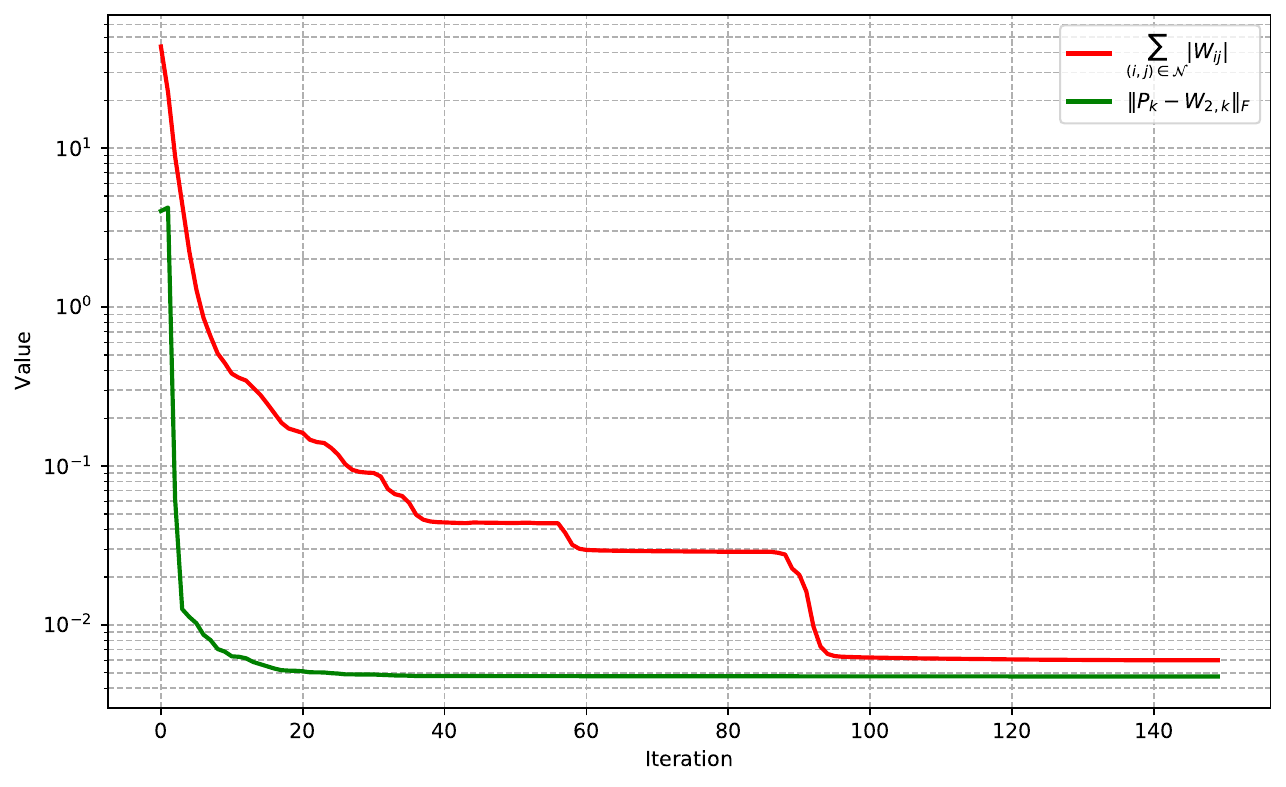}
        \caption{Feasibility of ADMM for Example \ref{exp2new}}
        \label{ADMM-high}
    \end{subfigure}
    \hfill
    \begin{subfigure}[b]{0.45\textwidth}
        \includegraphics[width=\textwidth]{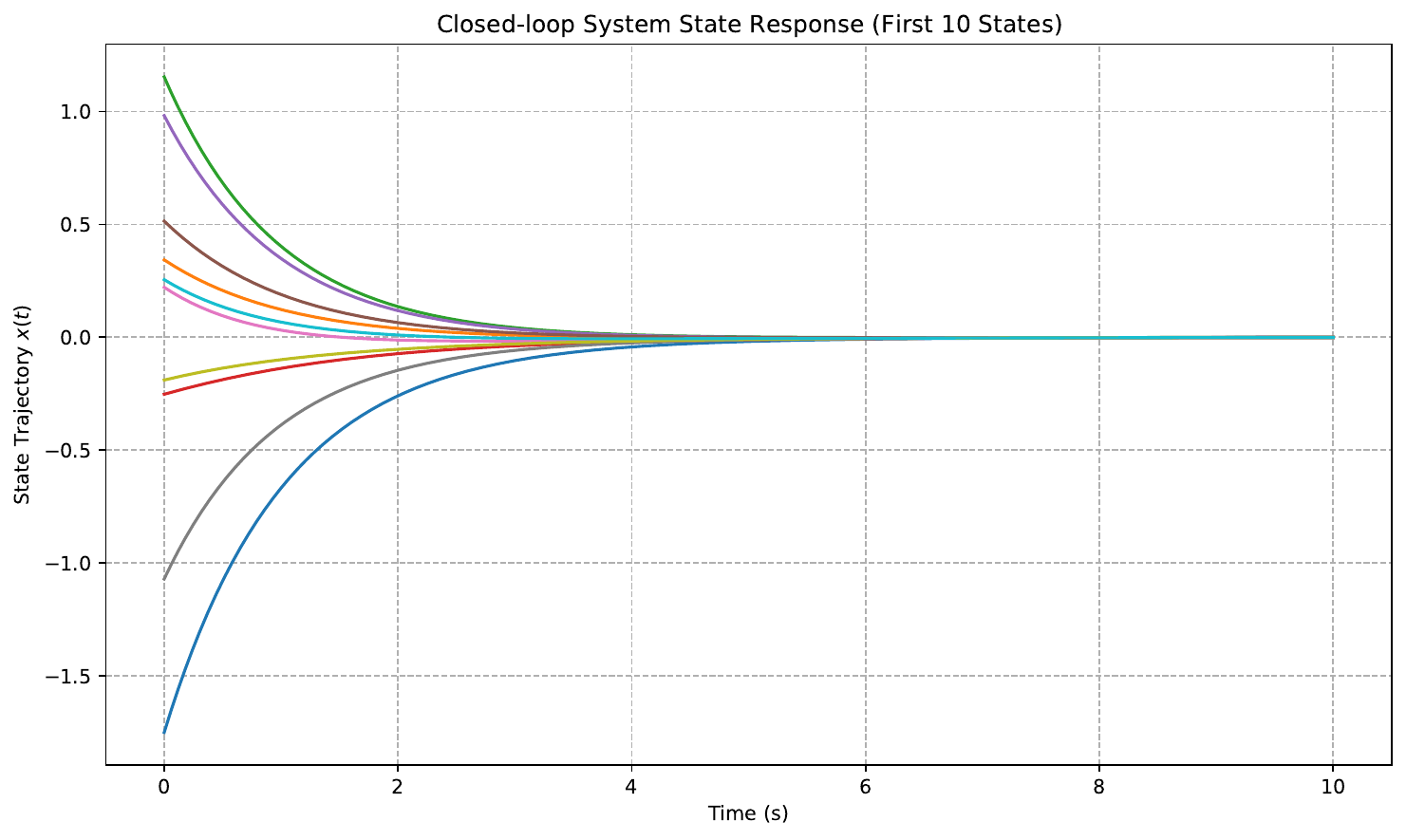}
        \caption{System Response}
        \label{response2}
    \end{subfigure}
        \caption{Figures of Example \ref{exp1} and \ref{exp2new}}
    \label{union1}
\end{figure*}

\section{Conclusion}

In this paper, we studied the DFT-LQ and SF-LQ problems. By reformulating these problems from an epi-composition perspective, we developed a direct approach to handle constraints without relying on penalty terms. We analyzed the convergence of the DR splitting, equivalently ADMM, under mild assumptions. When these fail, we proposed a PSGD algorithm combined with a DC relaxation. The proposed framework enables the design of group-sparse feedback gains with theoretical guarantees, without resorting to convex surrogates or restrictive structural assumptions. 
In future work, we plan to extend our framework to cases involving time-varying communication topologies and time-varying system matrices.

\section*{Appendix}

\subsection*{Proof of Lemma \ref{subgrad_varphi1}}

\begin{proof}
\textbf{(i) Proof of 1) and 2)}

We first prove that $\varphi_1=\mathcal{A}f$ is closed. 
Given $\tilde{s}\in\mathrm{dom}(\varphi_1)$, we consider a sequence $\{s_k\}_{k\geq 0}\subseteq \mathbf{lev}_{\leq\alpha}(\varphi_1)$ for some $\alpha\geq 0$ and such that
$s_k\to\tilde{s}$. Then, it suffices to show $\varphi_1(\tilde{s})\leq\alpha$. By the definition of $\varphi_1$, we may find a sequence $\{\widetilde{W}_k\}_{k\geq 0}$ s.t.
\begin{equation*}
  \mathcal{A}\widetilde{W}_k=s_k~\text{and}~f(\widetilde{W}_k)\leq \varphi_1(s_k)+\epsilon_k
\end{equation*}
with $\epsilon_k\downarrow 0$. 
By the definition of $f$ and Assumption \ref{ass1}, it follows that
\begin{align*}
  f(\widetilde{W})&\geq\langle {\rm vec}(R),\widetilde{W}\rangle=\langle R,W\rangle\geq \lambda_{{\rm min}}(R){\rm tr}(W)\\
  &\geq \lambda_{{\rm min}}(R)\Vert W\Vert_2;
\end{align*}
hence, $f$ is bounded below, while $\{\widetilde{W}_k\}_{k\geq 0}$ is bounded. Thus, without loss of generality, we may assume $\widetilde{W}_k\to \widetilde{W}^*$ and $\mathcal{A}\widetilde{W}^*=\tilde{s}$. By the closedness of $f$, it is obvious that
\begin{equation*}
  \varphi_1(\tilde{s})\leq f(\widetilde{W}^*)\leq \liminf_{k\to \infty} f(\widetilde{W}_k)\leq \liminf_{k\to\infty} \varphi_1(s_k)+\epsilon_k\leq \alpha.
\end{equation*} 
Therefore, $\mathcal{A}f$ is a closed function.  By Theorem 2.4.2 of \cite{hiriart2004fundamentals}, $\mathcal{A}f$ is convex. By definition, $s\in \partial (\mathcal{A}f)(\hat{s})$ iff
\begin{equation*}
  (\mathcal{A}f)(\hat{s}^\prime)\geq (\mathcal{A}f)(\hat{s})+\langle s,\hat{s}^\prime-\hat{s}\rangle,~~\forall \hat{s}^\prime,
\end{equation*}
which may be rewritten as
\begin{equation*}
    (\mathcal{A}f)(\hat{s}^\prime)\geq f(\widetilde{W})+\langle s,\hat{s}^\prime-\mathcal{A}\widetilde{W}\rangle,~~\forall \hat{s}^\prime
\end{equation*}
with $\widetilde{W}$ an arbitrary element in $\widetilde{W}(\hat{s})$. Additionally, due to the fact that $\mathcal{A}$ is surjective, the last relation is equivalent to
\begin{align*}
  f(\widetilde{W}^\prime)&\geq f(\widetilde{W})+\langle s,\mathcal{A}\widetilde{W}^\prime-\mathcal{A}\widetilde{W}\rangle\\
  &=f(\widetilde{W})+\langle \mathcal{A}^\top s,\widetilde{W}^\prime-\widetilde{W}\rangle,~~\forall \widetilde{W}^\prime,
\end{align*}
which means that $\mathcal{A}^\top s\in\partial f(\widetilde{W})$. By Theorem 23.8 of \cite{rockafellar1997convex}, it follows
\begin{equation*}
  \partial f(\widetilde{W})\supseteq \mathrm{vec}(R)+\mathcal{N}_{\Gamma_+^p}(\widetilde{W})+\sum_{i=1}^{M}\mathcal{N}_{\mathcal{H}_i^{-1}(\Gamma_+^n)}(\widetilde{W}).
\end{equation*}

\textbf{(ii) Proof of 3)}

Consider the following optimization problem
\begin{equation*}
 \begin{aligned}
\min_{\widetilde{P}}~~ &\gamma\Vert \pi(\widetilde{P})\Vert_0\\
{\rm s.t.}~~&\mathcal{B}\widetilde{P}=
\begin{bmatrix}
  \mathbf{0}_{N^*\times 1} \\
  -\widetilde{P}
\end{bmatrix}
=-s\doteq -
\begin{bmatrix}
  s_1 \\
  s_2
\end{bmatrix},
\end{aligned}
\end{equation*}
where $s_1= [I_{N^*}~\mathbf{0}_{N^*\times mn}]s$ and $s_2=[\mathbf{0}_{mn\times N^*}~I_{mn}]s$. By the definition of epi-composition function $\mathcal{B}g(-s)$, the global minimizer of the above problem is $\varphi_2(s)$, while by the straightforward computation, we have
\begin{equation*}
  \varphi_2(s)=\gamma \Vert \pi(-[\mathbf{0}_{mn\times N^*}~I_{mn}]s)\Vert_0+\delta_{A(s)}(s)
\end{equation*}
with $A(s)=\{s\colon [I_{N^*}~\mathbf{0}_{N^*\times mn}]s=0\}$. Since both $\ell_0$-norm and indicator function $\delta_{A(s)}(s)$ are closed, we can deduct that $\varphi_2(s)$ is closed.

\textbf{(ii) Proof of 4)}

Consider the following optimization problem
\begin{equation*}
 \begin{aligned}
\min_{\widetilde{W}}~~ &\langle {\rm vec}(R),\widetilde{W}\rangle\\
{\rm s.t.}~~&\mathcal{A}\widetilde{W}=s,~~\widetilde{W}\in\Gamma_+^p,~~\widetilde{\Psi}\in\Gamma_+^n,
\end{aligned}
\end{equation*}
while by the definition of epi-composition function, the global minimizer of the above problem is given by $\varphi_1(s)$. Clearly, the above problem is equivalent to
\begin{equation*}
 \begin{aligned}
\min_{\{W_{1,i}\}_{i=1}^t,W_3} &\langle R,W\rangle\\
{\rm s.t.}~~~~~&W=\begin{bmatrix}
              W_1 & \mathrm{vec}^{-1}(s_2)^\top \\
               \mathrm{vec}^{-1}(s_2) & W_3
             \end{bmatrix}\in\mathbb{S}_+^p,\\
&AW_1-B_2\mathrm{vec}^{-1}(s_2)+W_1A^\top-\mathrm{vec}^{-1}(s_2)^\top B_2^\top\\
&+B_1B_1^\top \in\mathbb{S}^n_-,\\
&W_1=\mathbf{blockdiag}\{W_{1,1},\dots,W_{1,t}\}+\mathbf{mat}(s_1).
\end{aligned}
\end{equation*}
Here, $\mathbf{blockdiag}\{W_{1,1},\dots,W_{1,t}\}$ denotes the block diagonal matrix with diagonal entries $W_{1,1},\dots,W_{1,t}$, and $\mathbf{mat}(s_1)$ denotes inserting $s_1$ into the off-block-diagonal positions of the matrix $W_1$, according to the definition of the matrix $\widehat{\mathcal{A}}$. Obviously, the negative semi-definite cone constraint is equivalent to the following Lyapunov matrix equation
\begin{equation*}
\begin{aligned}
(A-B_{2}\mathrm{vec}^{-1}(s_2) W_1^{-1})W_1&+W_1(A-B_{2}\mathrm{vec}^{-1}(s_2) W_1^{-1})^\top\\
  &+B_1B_1^\top\preceq 0.
\end{aligned}
\end{equation*}
By utilizing the properties of the Lyapunov matrix equation and the Schur complement lemma, the nonemptiness of the feasible set is equivalent to the solvability of:
\begin{equation*}
  A-B_2 \mathrm{vec}^{-1}(s_2) W_1^{-1} \text{~is Hurwitz}.
\end{equation*}
Given that the block-diagonal entries $\{W_{1,i}\}_{i=1}^t$ of $W_1$ are free variables, 
by Gershgorin circle theorem and choosing $\{W_{1,i}\}_{i=1}^t$ to be sufficiently large, the minimum eigenvalue of $W_1$ can be made arbitrarily large.  As a result, $B_2 \mathrm{vec}^{-1}(s_2) W_1^{-1}$ can be made arbitrarily small in norm. Due to the continuity of eigenvalues with respect to matrix perturbations, it follows that $A - B_2 \mathrm{vec}^{-1}(s_2) W_1^{-1}$ remains Hurwitz since $A$ is Hurwitz.
Hence, the domain $\mathrm{dom}(\varphi_1)$ equals $\mathbb{R}^{N^*+mn}$ provided that $A$ is a Hurwitz matrix.

Conversely, it can be shown that if $A$ is not a Hurwitz matrix, then $\mathrm{dom}(\varphi_1)$ is a strict subset of $\mathbb{R}^{N^*+mn}$. Assume
\begin{equation*}
  P^{-1}AP=\begin{bmatrix}
J_1 & & \\
& \ddots & \\
& & J_k
\end{bmatrix}.
\end{equation*}
In this decomposition, $J_1, \dots, J_k$ are the Jordan blocks associated with the eigenvalues $\Re(\lambda_1) \leq \dots \leq \Re(\lambda_k)$ of the matrix $A$. The existence of an invertible matrix $P$ is guaranteed by the Jordan decomposition theorem. Denote
\begin{equation*}
  P^{-1}B_2\doteq\begin{bmatrix}
                   b_1 \\
                   b_2 \\
                   \vdots\\
                   b_n
                 \end{bmatrix},
~~\mathrm{vec}^{-1}(s_2)\doteq
\begin{bmatrix}
  w_1 & w_2 &\cdots & w_n
\end{bmatrix}
\end{equation*}
with $\forall i\in[n], b_i\in\mathbb{R}^{1\times m},w_i\in\mathbb{R}^m$. Owing to the arbitrariness of $\{w_i\}_{i=1}^n$, we may, without loss of generality, assume that $w_i \in \ker(b_n)$ for all $i = 1, \dots, n$. Then, we have
\begin{equation*}
  P^{-1}B_2\mathrm{vec}^{-1}(s_2)=\begin{bmatrix}
      \ast  \\
      0
   \end{bmatrix},
\end{equation*}
and thus
\begin{equation*}
P^{-1}(A-B_2 \mathrm{vec}^{-1}(s_2) W_1^{-1} )P=\begin{bmatrix}
* & *& *& *\\
0 &\cdots & 0 &\lambda_k
\end{bmatrix}.
\end{equation*}
Since $\Re(\lambda_k) > 0$, it follows that $\mathrm{dom}(\varphi_1) \neq \mathbb{R}^{N^*+mn}$.
\end{proof}

\subsection*{Proof of Theorem \ref{local-minimal}}

\begin{lemma}\label{regular}
For any closed convex set $\mathcal{S}$, let $t(s_2)\doteq \gamma \Vert \pi(s_2)\Vert_0+\delta_{\mathcal{S}}(s_2)$. It holds that $t$ is regular at any $s_2\in\mathbb{R}^{mn}$.
\end{lemma}

\begin{proof}
For any $w\in [st]$, we denote 
$$P_\w\doteq \{s_2\colon [\pi(s_2)]_i=0,\forall i\in w^c\}.$$
Denoting $\w= \mathrm{supp}(\pi(s_2))$, the regular subdifferential of $t$ at $s_2$ exhibits the following form
\begin{align*}
	&~~\widehat{\partial} t(s_2)  \\
&=\left\{ \theta \colon \liminf_{\underset{s_2^\prime \neq s_2}{s_2^\prime \rightarrow s_2} } \frac{t(s_2^\prime) - t(s_2) - \langle \theta, s_2^\prime-s_2\rangle }{\| s_2^\prime-s_2\|}\geq 0 \right\} \\
& = \left\{ \theta\colon \liminf_{\underset{s_2^\prime \neq s_2}{s_2^\prime \rightarrow s_2, t(s_2^\prime) = t(s_2) } } \frac{ - \langle \theta, s_2^\prime-s_2\rangle}{\| s_2^\prime-s_2 \|} \geq 0\right\} \\
& =  \left\{\theta\colon \liminf_{\underset{s_2^\prime \neq s_2}{s_2^\prime \rightarrow s_2, s_2^\prime \in \mathcal{S} \cap P_{\w} } } \frac{- \langle \theta, s_2^\prime-s_2\rangle}{\| s_2^\prime-s_2 \|}\geq 0  \right\}\\
 & = \left\{ \theta\colon  \liminf_{\underset{s_2^\prime \neq s_2}{s_2^\prime \rightarrow s_2, } } \frac{ \delta_{\mathcal{S} \cap P_{\w}}(s_2^\prime) -\delta_{\mathcal{S} \cap P_{\w} }(s_2) - \langle \theta, s_2^\prime-s_2 \rangle}{\|s_2^\prime-s_2 \|}\geq 0 \right\} \\
 & =  \widehat{\partial} \delta_{\mathcal{S} \cap P_{\w}} (s_2).
\end{align*}
Here, if $s_2^\prime\to s_2$ and $\mathrm{supp}(\pi(s_2^\prime))\supsetneq \mathrm{supp}(\pi(s_2))$, then $t(s_2^\prime)\geq t(s_2)+\gamma$; hence the second equality holds.
By the same token, the limiting subdifferential of $t$ at $s_2$ is
\begin{align*}
\partial t(s_2)=& \left \{ \theta\colon \exists\, s_2^\prime \rightarrow s_2, t(s_2^\prime) = t(s_2), \widehat{\partial}t(s_2^\prime) \ni \theta^\prime  \rightarrow \theta \right\}\\
=& \partial \delta_{\mathcal{S}\cap P_{\w}}(s_2).
\end{align*}
Hence, by Proposition 8.12 of \cite{rockafellar-book}, $\widehat{\partial} t(s_2)=\partial t(s_2)$. Again, by the same argument, the horizon subdifferential of $t_2$ at $s_2$ admits:
\begin{align*}
\partial^{\infty} t(s_2) =  \partial^{\infty} \delta_{\mathcal{S}\cap P_{\w}}(s_2) = \partial \delta_{\mathcal{S} \cap P_{\w}}(s_2) ^{\infty}  = \partial t(s_2) ^{\infty},
\end{align*}
where for a set $C$, $C^\infty$ denotes the horizon cone of $C$; see Definition 3.3 of \cite{rockafellar-book}.
In summary, by Corollary 8.11 of \cite{rockafellar-book}, $t$ is regular at $s_2$.
\end{proof}

We are now ready to complete the proof of Theorem \ref{local-minimal}.

\begin{proof}
In fact, the theorem is a direct corollary of a more general result established here, which applies to problem 
\begin{equation}\label{pf1}
  \min_{s_2\in\mathcal{S}}~\zeta(s_2)
\end{equation}
with arbitrary closed convex $\mathcal{S}$. For any subset $\w \subseteq [st]$, define the axillary problem
\begin{equation}\label{pf2}
  \begin{aligned}
  \min_{s_2\in\mathcal{S}}~&\widehat{\varphi}_1(s_2),\\
  \text{s.t.}~& [\pi(s_2)]_i=0,~i\in\w^c,
  \end{aligned}
  \tag{$Q_{\w}$}
\end{equation} 
and denote $\mathcal{L}\doteq\{s_2\colon s_2\text{~is a local minimizer of problem \eqref{pf1}}\}$.

First, we claim that
\begin{equation}\label{pf3}
  \mathcal{L}=\mathcal{L}^\prime\doteq\bigcup_{\w \subseteq [st]} \{s_2\colon s_2 \text{~is a global minimizer of \eqref{pf2}}\}.
\end{equation}
Indeed, it is straightforward to show $\mathcal{L}^\prime\subseteq \mathcal{L}$; hence it suffices to prove $\mathcal{L}\subseteq\mathcal{L}^\prime$. For any ${s}_2^*\in\mathcal{L}$, ${s}_2^*$ is a local minimizer of the following convex problem
\begin{equation}\label{pf4}
  \min_{s_2\in\mathcal{S}}~\widehat{\varphi}_1(s_2),~\text{s.t.}~[\pi(s_2)]_i=0,\forall i\in (\w^*)^c
\end{equation}
with $\w^*\doteq \mathrm{supp}(\pi({s}_2^*))$. Problem \eqref{pf4} is convex by Lemma \ref{subgrad_varphi1}; hence ${s}_2^*$ is also a global minimizer of \eqref{pf4} and claim \eqref{pf3} holds.

Defining $\bar{\w}\doteq \mathrm{supp}(\pi(\bar{s}_2))$, according to Lemma \ref{regular}, it holds that
\begin{align*}
  0 &\in \partial (\widehat{\varphi}_1(\bar{s}_2)+\gamma \Vert \pi(\bar{s}_2)\Vert_0+\delta_{\mathcal{S}}(\bar{s}_2))\\
  &= \partial \widehat{\varphi}_1(\bar{s}_2) + \partial \delta_{\mathcal{S}\cap P_{\bar{\w}}}(\bar{s}_2)\\
  &=\partial (\widehat{\varphi}_1(\bar{s}_2) +  \delta_{\mathcal{S}\cap P_{\bar{\w}}}(\bar{s}_2)),
\end{align*}
where the equalities hold by Corollary 10.9 of \cite{rockafellar-book}. Since problem \eqref{pf2} is convex, $\bar{s}_2$ is a global minimizer of \eqref{pf2}; hence by \eqref{pf3}, $\bar{s}_2$ is a local minimizer of \eqref{pf1}.
\end{proof}

\subsection*{Proof of Theorem \ref{smoothDR}}

Define function $\widehat{\varphi}_\eta\colon \mathbb{R}^{mn}\to\mathbb{R}$:
\begin{align*}
  \widehat{\varphi}_\eta(s_2)\doteq \min_{w\in\mathcal{S}^*}&\bigg\{ \widehat{\varphi}_2(w)+\widehat{\varphi}_1(u)+\langle \nabla \widehat{\varphi}_1(u),w-u\rangle\\
  &~~+\frac{1}{2\eta}\Vert w-u\Vert^2\bigg\}
\end{align*}
with $u\doteq \mathbf{prox}_{\eta \widehat{\varphi}_1}(s_2)$. By Assumption \ref{lsmooth}, $\nabla \widehat{\varphi}_1(s_2)$ is well-defined for any $s_2\in\mathcal{S}^*$.
We first prove the following sufficient descent lemma.

\begin{lemma}\label{desecent_lemma}
There exists positive constant $c$ satisfying
\begin{equation}\label{descentlemmabound}
  \widehat{\varphi}_\eta(s_{2;k})-\widehat{\varphi}_\eta(s_{2;k+1})\geq \frac{c}{(1+\eta L_{\widehat{\varphi}_1})^2}\Vert s_{2;k}-s_{2;k+1}\Vert^2.
\end{equation}
\end{lemma}

\begin{proof}
Combining the definition of proximal operator, $u_{k+1}\in\mathbf{prox}_{\eta\widehat{\varphi}_1}(s_{2;k+1})$ and according to Proposition 2.3 of \cite{themelis2020douglas}, we can deduce that $s_{2;k+1}=u_{k+1}+\eta\nabla \widehat{\varphi}_1(u_{k+1})$. Upon
\begin{equation*}
\begin{aligned}
  \widehat{\varphi}_\eta(s_{2;k+1})= \min_{w\in\mathcal{S}^*}&\bigg\{ \widehat{\varphi}_2(w)+\widehat{\varphi}_1(u_{k+1})+\langle \nabla \widehat{\varphi}_1(u_{k+1}),\\
  &~~w-u_{k+1}\rangle+\frac{1}{2\eta}\Vert w-u_{k+1}\Vert^2\bigg\},
\end{aligned}
\end{equation*}
it can be readily verified that $v_{k+1}$ is an optimal solution of the above problem. For $x,y\in\mathcal{S}^*$, let
\begin{equation*}
  \rho(y,x)=\frac{\sigma L}{2(L+\sigma)}\Vert y-x\Vert^2+\frac{1}{2(L+\sigma)}\Vert \nabla p_1(y)-\nabla p_1(x) \Vert^2
\end{equation*}
with $L\geq L_{\widehat{\varphi}_1}$ and $\sigma\in[-L,-L_{\widehat{\varphi}_1}]$. Thus, it follows that
  \begin{align*}
  &~~~~\widehat{\varphi}_\eta(s_{2;k+1})\\
  &\leq  \widehat{\varphi}_2(v_{k})+\widehat{\varphi}_1(u_{k+1})+\langle \nabla \widehat{\varphi}_1(u_{k+1}),v_{k}-u_{k+1}\rangle\\
  &~~~~+\frac{1}{2\eta}\Vert v_{k}-u_{k+1}\Vert^2\\
  &\overset{\mathrm{(\romannumeral1)}}{\leq}-\rho(u_{k},u_{k+1})+\langle \nabla \widehat{\varphi}_1(u_{k+1}),v_{k}-u_{k}\rangle\\
  &~~~~+\widehat{\varphi}_1(u_{k+1})+\widehat{\varphi}_2(v_{k})+\frac{1}{2\eta}\Vert v_{k}-u_{k+1}\Vert^2\\
  &\overset{\mathrm{(\romannumeral2)}}{=}\widehat{\varphi}_\eta(s_{2;k})-\rho(u_{k},u_{k+1})\\
  &~~~~+\langle \nabla \widehat{\varphi}_1(u_{k+1})-\nabla \widehat{\varphi}_1(u_{k}),v_{k}-u_{k}\rangle\\
  &~~~~\frac{1}{2\eta}\Vert u_{k}-u_{k+1}\Vert^2+\frac{1}{\eta}\langle u_{k+1}-u_{k},u_{k}-v_{k}\rangle,
  \end{align*}
where inequality (\romannumeral1) holds due to the $L_{\widehat{\varphi}_1}$-smoothness of $\widehat{\varphi}_1$, and equality (\romannumeral2) follows from the definition of $\widehat{\varphi}_\eta$. By Proposition 2.3 of \cite{themelis2020douglas} and Assumption \ref{lsmooth}, $u_{k}\in\mathbf{prox}_{\eta\widehat{\varphi}_1}(s_{2;k})$ implies $s_{2;k}=u_{k}+\eta \nabla\widehat{\varphi}_1(u_{k})$; by the same token, 
we know that $s_{2;k+1}=u_{k+1}+\eta \nabla\widehat{\varphi}_1(u_{k+1})$. Thus, it follows that 
$$\frac{1}{\xi}(s_{2;k}-s_{2;k+1})=\frac{1}{\xi}(u_{k}-u_{k+1})+\frac{\eta}{\xi}(\nabla \widehat{\varphi}_1(u_{k})-\nabla \widehat{\varphi}_1(u_{k+1})).$$
Therefore, we have
\begin{equation*}
\begin{aligned}
  &\widehat{\varphi}_\eta(s_{2;k})-\widehat{\varphi}_\eta(s_{2;k+1})
  \geq\frac{2-\xi}{2\eta\xi}\Vert u_{k}-u_{k+1}\Vert^2\\
  &-\frac{\eta}{\xi}\Vert \nabla \widehat{\varphi}_1(u_{k})-\nabla \widehat{\varphi}_1(u_{k+1})\Vert^2+\rho(u_{k},u_{k+1}).
\end{aligned}
\end{equation*}
Based on the definition of $\rho$, we know
\begin{equation*}
  \widehat{\varphi}_\eta(s_{2;k})-\widehat{\varphi}_\eta(s_{2;k+1})\geq c\Vert u_{k}-u_{k+1}\Vert^2.
\end{equation*}
Since $\mathbf{prox}_{\eta \widehat{\varphi}_1}$ is $\frac{1}{\eta L_{\widehat{\varphi}_1}}$-strongly monotone, \eqref{descentlemmabound} holds. It remains to prove $c>0$, which follows by an argument analogous to that in the proof of Theorem 4.1 in \cite{themelis2020douglas}, and is therefore omitted here.
\end{proof}

With the above lemma established, we now return to the proof of Theorem \ref{smoothDR}.

\begin{proof}
To exclude the trivial case, we assume without loss of generality that $u_k \neq v_k$ for all $k$.
By Lemma \ref{desecent_lemma}, it holds that 
  \begin{align*}
    &\frac{c\xi}{(1+\eta L_{\widehat{\varphi}_1})^2}\sum_{k\geq 0}\Vert u_{k}-v_{k}\Vert^2\\
=&\frac{c}{(1+\eta L_{\widehat{\varphi}_1})^2}\sum_{k\geq 0}\Vert s_{2;k}-s_{2;k+1}\Vert^2\\
\leq&  \sum_{k\geq 0}(\widehat{\varphi}_\eta(s_{2;k})-\widehat{\varphi}_\eta(s_{2;k+1}))\\
\leq& \widehat{\varphi}_\eta(s_{2;0}) - \inf_{s_2\in\mathcal{S}^*} \widehat{\varphi}_\eta(s_2)
  \end{align*}
  with $c>0$. By Assumption \ref{lsmooth}, $\widehat{\varphi}_1$ is $L_{\widehat{\varphi}_1}$-smooth on $\mathcal{S}^*$; hence, we can show that
  \begin{align*}
  &\widehat{\varphi}_\eta(s_2)\\
=& \min_{w\in\mathcal{S}^*}\bigg\{ \widehat{\varphi}_2(w)+\widehat{\varphi}_1(u)+\langle \nabla \widehat{\varphi}_1(u),w-u\rangle+\frac{1}{2\eta}\Vert w-u\Vert^2\bigg\}\\
\geq& \min_{w\in\mathcal{S}^*}\bigg\{ \widehat{\varphi}_2(w)+\widehat{\varphi}_1(w)-\frac{L_{\widehat{\varphi}_1}}{2}\Vert w-u\Vert^2+\frac{1}{2\eta}\Vert w-u\Vert^2\bigg\}\\
=&\min_{w\in\mathcal{S}^*}\bigg\{ \widehat{\varphi}_2(w)+\widehat{\varphi}_1(w)+\underbrace{\left(\frac{1}{2\eta}-\frac{L_{\widehat{\varphi}_1}}{2}\right)}_{\geq 0\text{~by the selection rule of~}\eta}\Vert w-u\Vert^2\bigg\}\\
\geq &\min_{w\in\mathcal{S}^*}\bigg\{ \widehat{\varphi}_2(w)+\widehat{\varphi}_1(w)\bigg\}\geq 0
\end{align*}
with $u=\mathbf{prox}_{\eta \widehat{\varphi}_1}(s_2)$. Therefore, we can deduce that $\inf\limits_{s_2\in\mathcal{S}^*} \widehat{\varphi}_\eta(s_2)>-\infty$, $\{u_k-v_k\}_{k\geq 1}$ is square summable and $u_k-v_k\to 0$ as $k\to\infty$, which necessitates $\{u_k\}_{k\geq 0}$ and $\{v_k\}_{k\geq 0}$ share same cluster points. We suppose $\{u_k\}_{k\in\mathbb{K}}\to u^\prime$ for some $\mathbb{K}\subseteq \mathbb{N}$. Then, $\{v_k\}_{k\in\mathbb{K}}\to u^\prime$. As mentioned above, $u_{k}\in\mathbf{prox}_{\eta\widehat{\varphi}_1}(s_{2;k})$ implies $s_{2;k}=u_{k}+\eta \nabla\widehat{\varphi}_1(u_{k})$.
Henceforth, it follows that
\begin{align*}
  u^\prime = \lim_{k\in\mathbb{K},k\to\infty} v_k&{\in}\limsup_{k\in\mathbb{K},k\to\infty} \mathbf{prox} _{\eta\widehat{\varphi}_2}(2u_{k}-s_{2;k})\notag\\
  &\overset{\rm (i)}{\in}\limsup_{k\in\mathbb{K},k\to\infty} \mathbf{prox} _{\eta\widehat{\varphi}_2}(u_k-\eta\nabla\widehat{\varphi}_1(u_k))\notag\\
  &\overset{\rm (ii)}{\subseteq} \mathbf{prox} _{\eta\widehat{\varphi}_2}(u^\prime-\eta\nabla\widehat{\varphi}_1(u^\prime)).
  \end{align*}
Here, (i) holds by $s_{2;k}=u_{k}+\eta \nabla\widehat{\varphi}_1(u_{k})$, while (ii) holds by the outer semi-continuity of $\mathbf{prox} _{\eta\widehat{\varphi}_2}$; see Example 5.23(b) of \cite{rockafellar-book} (and $\ell_0$-norm is naturally proximally bounded with infinitely large threshold).
Therefore, by Theorem 10.1 of \cite{rockafellar-book}, we obtain
\begin{equation*}
  0\in\nabla \widehat{\varphi}_1(u^\prime)+\widehat{\partial}\widehat{\varphi}_2(u^\prime).
\end{equation*}
We know
\begin{align*}
{\partial}\zeta (u^\prime)\supseteq\widehat{\partial}\zeta (u^\prime)&\overset{\mathrm{(\romannumeral1})}{\supseteq} \nabla \widehat{\varphi}_1(u^\prime)+\widehat{\partial}\delta_{\mathcal{S}^*}(u^\prime)+\widehat{\partial}\widehat{\varphi}_2(u^\prime)\\
&\overset{\mathrm{(\romannumeral2})}{=}\nabla \widehat{\varphi}_1(u^\prime)+{\partial}\delta_{\mathcal{S}^*}(u^\prime)+\widehat{\partial}\widehat{\varphi}_2(u^\prime)\\
&\overset{\mathrm{(\romannumeral3})}{=}\nabla \widehat{\varphi}_1(u^\prime)+\mathcal{N}_{\mathcal{S}^*}(u^\prime)+\widehat{\partial}\widehat{\varphi}_2(u^\prime)\\
&\overset{\mathrm{(\romannumeral4})}{\ni} 0,
\end{align*}
where by Corollary 10.9 of \cite{rockafellar-book}, inequality (\romannumeral1) holds.
Since $\mathcal{S}^*$ is a closed convex set, the operator $\partial \delta_{\mathcal{S}^*}(\cdot) = \widehat{\partial} \delta_{\mathcal{S}^*}(\cdot)$, and thus equality (\romannumeral2) holds.
By Example 3.5 of \cite{beck2019optimization}, equality (\romannumeral3) holds.
From the definition of normal cone $\mathcal{N}_{\mathcal{S}^*}(\cdot)$, we know that $0 \in \mathcal{N}_{\mathcal{S}^*}(u^\prime)$.
Combining that $\widehat{\varphi}_1$ and $\widehat{\varphi}_2$ are semi-algebraic (by Example 2 of \cite{bolte2014proximal}), the theorem is proved by Theorem 4.4 of \cite{themelis2020douglas}.
\end{proof} 

\subsection*{Proof of Theorem \ref{thm_app}}

\begin{proof}
Obviously, we have 
$$\vartheta_{\alpha_{k+1},h}(s_2)\geq\vartheta_{\alpha_k,h}(s_2), ~~\forall s_2\in\R^{mn}.$$ 
Hence, it follows that 
$$\zeta_{\alpha_{k+1},h}(s_2)\geq\zeta_{\alpha_k,h}(s_2),~~\forall s_2\in\R^{mn}.$$ 
By Proposition 7.4 of \cite{rockafellar-book}, it follows that
\begin{equation*}
  \zeta_{\alpha_k,h}(s_2) \xrightarrow{e} \left(\sup_{k}\mathbf{cl}[\zeta_{\alpha_k,h}]\right)(s_2),
\end{equation*}
where the definition of epi-convergence and lower closure operator $\mathbf{cl}[\cdot]$ can be found in Definition 7.1 and Page 14 of \cite{rockafellar-book}, respectively.
Remarkably, there exists a constant $c_{f,\alpha_k}$, such that  $\lim\limits_{\alpha_k\to 0} \vartheta_{\alpha_k,h}(s_2)=c_{h,\alpha_k} \Vert \pi(s_2) \Vert_0$. 
Thus, it follows that $$\left(\sup_{k}\mathbf{cl}[\zeta_{\alpha_k,h}]\right)(s_2)=\zeta(s_2),~~\forall s_2\in\R^{mn}.$$ 
In fact, for every $\alpha_k$, $\zeta_{\alpha_k,h}(s_2)$ is a coercive function.
Indeed, if $\Vert s_{2,n}\Vert\to\infty$, letting
\begin{equation*}
  \widetilde{W}^*_n\in\argmin_{\mathcal{A}\widetilde{W}=Us_{2,n}}~f(\widetilde{W})=\langle \widetilde{R},\widetilde{W}\rangle+\text{indicator functions},
\end{equation*}
then $\Vert\widetilde{W}^*_n\Vert_F\to\infty$. By Assumption \ref{ass1}, we know $f(\widetilde{W}_n^*)\to\infty$; hence $\widehat{\varphi}_1(s_{2,n})\to\infty$ and $\zeta_{\alpha_k,h}$ is coercive.
According to Exercise 7.32 of \cite{rockafellar-book}, the sequence $\{\zeta_{\alpha_k,h}(s_2)+\delta_{\R^{mn}}(s_2)\}_{k\geq 0}$ is eventually level-bounded. By Lemma \ref{subgrad_varphi1}, $\zeta_{\alpha_k,h}$ and $\zeta$ are closed and proper; hence, we finish the proof by Theorem 7.33 of \cite{rockafellar-book}.
\end{proof}

\subsection*{Proof of Theorem \ref{thm7}}

\begin{proof}
First, by Theorem 2.21 of \cite{ref5} and $\Xi_\epsilon\subseteq \mathrm{ri}(\mathrm{dom}(\widehat{\varphi}_1))$, $\zeta_{\alpha,h}(s_2)$ is locally Lipschitz for any $s_2\in \cap \Xi_\epsilon$. Moreover, the epigraph of $\varphi_1$  (denoted as $\mathbf{epi}(\varphi_1)$) admits the following structure:
\begin{equation*}
  \mathbf{epi}(\varphi_1)= \left\{ (s,t) \left|
  \begin{aligned}
  &\exists \widetilde{W}:\widetilde{W}\in\Gamma_+^p,\widetilde{\Psi}_i\in\Gamma_+^n,~i\in[M],\\
  &\mathcal{A}\widetilde{W}=s, \langle \mathrm{vec}(R),\widetilde{W}\rangle \leq t
  \end{aligned}
  \right.
  \right\},
\end{equation*}
which is semi-algebraic by Tarski–Seidenberg Theorem. In addition, by Example 2 of \cite{bolte2014proximal}, $\widehat{\varphi}_1$ is semi-algebraic and hence definable in the o-minimal structure (see Definition 5.10 of \cite{davis2020stochastic}); while, by the same token, $\Xi_\epsilon$ is definable in the o-minimal structure for any choice of $\epsilon$.

Moreover, it follows that
\begin{align*}
  &\partial {\zeta}_{\alpha,h}(s_{2})\\
  =&\partial \left(\widehat{\varphi}_1(s_2)+\frac{\gamma}{c_{h,\alpha}}(h(\widetilde{\pi}(s_2))-\mathbf{env}_\alpha(h)(\widetilde{\pi}(s_2)))\right)\\
  \overset{\rm (i)}{=}&\partial \left(\widehat{\varphi}_1(s_2)+\frac{\gamma}{c_{h,\alpha}}h(\widetilde{\pi}(s_2))\right)-\frac{\gamma}{c_{h,\alpha}} \nabla \mathbf{env}_\alpha(h)(\widetilde{\pi}(s_2))\\
  \overset{\rm (ii)}{=}&\partial \widehat{\varphi}_1(s_2)+\frac{\gamma}{c_{h,\alpha}}\partial h(\widetilde{\pi}(s_2))-\frac{\gamma}{c_{h,\alpha}} \nabla \mathbf{env}_\alpha(h)(\widetilde{\pi}(s_2))\\
  \overset{\rm (iii)}{=}&\underbrace{U^\top\partial {\varphi}_1(Us_2)+\frac{\gamma}{c_{h,\alpha}}\partial h(\widetilde{\pi}(s_2))}_{\text{Minkowski sum of convex sets}}-\underbrace{\frac{\gamma}{c_{h,\alpha}} \nabla \mathbf{env}_\alpha(h)(\widetilde{\pi}(s_2))}_{\text{offset term}}.
\end{align*}
Here, by Theorem 6.60 of \cite{ref5}, $\mathbf{env}_\alpha(h)\circ \widetilde{\pi}$ is strictly differentiable at any $s_2\in \mathbb{R}^{mn}$; thus, by Theorem 9.18 of \cite{rockafellar-book}, $\partial^\infty\mathbf{env}_\alpha(h)\circ \widetilde{\pi}(s_2)=\{0\}$. Henceforth, Theorem 10.6 of \cite{rockafellar-book} implies equality (i). By Lemma \ref{subgrad_varphi1}, $\varphi_1$ is closed and convex, while equalities (ii) and (iii) hold due to Corollary 3.38 and  Theorem 3.43 of \cite{ref5}, respectively. Recalling that we have shown that $\partial \zeta_{\alpha,h}(s_2)$ can be rewritten as the Minkowski sum of convex sets with an offset term $\frac{\gamma}{c_{h,\alpha}} \nabla \mathbf{env}_\alpha(h)(\widetilde{\pi}(s_2))$, $\partial \zeta_{\alpha,h}(s_2)$ is a convex set and is equal to $\partial^C \zeta_{\alpha,h}(s_2)$

To finish the proof, by Theorem 4.2 of \cite{davis2020stochastic}, it suffices to show the descent property and weak Sard property (see Assumption F of \cite{davis2020stochastic} for the definitions). By Theorem 5.8 and Lemma 6.3 of \cite{davis2020stochastic} the descent property holds. To argue the weak Sard property, it suffices to follow the same approach as Corollary 6.4 of \cite{davis2020stochastic}.
\end{proof}

\section*{References}

\bibliographystyle{IEEEtran}
\bibliography{tac_7.9}

\end{document}